\documentclass[12pt,a4paper,twoside]{article}
\usepackage{fancyhdr}
\usepackage{enumerate,xcolor}

\usepackage[top=100pt,bottom=110pt,left=100pt,right=100pt]{geometry}

\usepackage{float}
\usepackage{caption}

\usepackage{amsmath}
\usepackage{amsthm}
\usepackage{amsfonts}
\usepackage{amssymb}
\usepackage{mathrsfs}
\usepackage{graphicx}
\usepackage[toc,page]{appendix}

\numberwithin{equation}{section}

\newtheorem{theorem}{Theorem}[section]
\newtheorem{lemma}{Lemma}[section]
\newtheorem{proposition}{Proposition}[section]
\theoremstyle{remark}
\newtheorem*{remark}{Remark}
\theoremstyle{definition}
\newtheorem*{definition}{Definition}
\newtheorem*{example}{Example}

\title{Reduced models of cardiomyocytes excitability:\\ comparing
Karma and FitzHugh-Nagumo}
\author{Maria Elena Gonzalez Herrero, Christian Kuehn\\ and Krasimira Tsaneva-Atanasova}

\begin{document}

\pagenumbering{arabic} 
\pagestyle{fancy} 
\fancyhf{}
\fancyfoot[CE,CO]{\thepage} 
\fancyhead[RO]{\leftmark} 
\fancyhead[LE]{\rightmark} 

\maketitle

\begin{abstract}
    Since Noble adapted in 1962 the model of Hodgkin and Huxley to fit Purkinje fibres the refinement of models for cardiomyocytes has continued. Most of these models are high-dimensional systems of coupled equations so that the possible mathematical analysis is quite limited, even numerically. This has inspired the development of reduced, phenomenological models that preserve qualitatively the main featuture of cardiomyocyte's dynamics. In this paper we present a systematic comparison of the dynamics between two notable low-dimensional models, the FitzHugh-Nagumo model \cite{FitzHugh55, FitzHugh60, FitzHugh61} as a prototype of excitable behaviour and a polynomial version of the Karma model \cite{Karma93, Karma94} which is specifically developed to fit cardiomyocyte's behaviour well. We start by introducing the models and considering their pure ODE versions. We analyse the ODEs employing the main ideas and steps used in the setting of geometric singular perturbation theory. Next, we turn to the spatially extended models, where we focus on travelling wave solutions in 1D. Finally, we perform numerical simulations of the 1D PDE Karma model varying model parameters in order to systematically investigate the impact on wave propagation velocity and shape. In summary, our study provides a reference regarding key similarities as well as key differences of the two models.
\end{abstract}

\section{Introduction}
    Excitability is a fundamental property of cardiomyocytes that generally plays a very important role in biology and medicine. Indeed, phenomena such as exchange of information between neurons, muscle contraction including cardiac arrhythmias, the emergence of organised patterns during development, all emerge as a result of excitable behaviour.\\

    Hodgkin and Huxley proposed the first ionic model to represent excitable behaviour, namely action potentials in a nerve fibre \cite{Hodgkin52}. This model is given by a 4-dimensional system of differential equations with one variable for the voltage and three gating variables for the ion channels. By adapting those equations to cardiac cells Noble developed in \cite{Noble62} a similar model for Purkinje cells opening up a new line in mathematical modelling focused on the heart. Since then there have been many different models either including different currents or ionic pumps (see for example further models developed by Noble et al.~\cite{McAllister75, DiFrancesco85}) and also models for other parts of the heart instead of Purkinje fibres like e.g. ventricular cells in the Beeler-Reuter model \cite{Beeler77}. To find a more extensive list of cardiac cell models see \cite{ScholarCellModels}.\\
    
    All the models mentioned above are quite complex with at least four variables and highly nonlinear. That makes analytical results often impossible, or at least extremely cumbersome. Furthermore, many models are even too complex for detailed numerical analysis and simulation for many parameters. This is why already in 1962 FitzHugh developed a simplified model of the Hodgkin-Huxley equations as a variation of the van der Pol oscillator \cite{vdP20,vdP26}, which focused on capturing the excitable properties of the system (see \cite{FitzHugh55, FitzHugh60, FitzHugh61}). Almost at the same time Nagumo published the corresponding electrical circuit in \cite{Nagumo62}. Their model reduces the four-dimensional system of Hodgkin and Huxley into two equations separating the fast time-scale of the excitation and the slow time-scale of recovery. The equations are given by
    \begin{equation}\label{eq:FHN}
		\begin{aligned}
			\frac{\partial v}{\partial t} &= D\Delta v + v-\frac{v^3}{3}-w+I\\
			\frac{\partial w}{\partial t} &= \varepsilon(v+a-bw)
		\end{aligned}
	\end{equation}
	where $v$ corresponds to the voltage and $w$ to a slow gating variable, $a,b$ are model parameters, $D$ is the diffusion coefficient and $I$ is an external current. Although its representation of nerve fibres or cardiac cells is not as precise as it would be with a more complex model, the FitzHugh-Nagumo (FHN) model has been studied extensively in the literature and is one of the prototype models for excitable media due to its simplicity.\\
    
    In the same spirit as FitzHugh, Karma developed in \cite{Karma93,Karma94} a reduced version of the Noble model preserving the fast-slow structure of FitzHugh-Nagumo. Here we present a systematic analysis of the Karma model comparing it to the FitzHugh-Nagumo model given that both models have been extensively used to model the behaviour of cardiomyocytes. Some additional references  can be found in \cite{Beck08, Mitchell03} for the Karma model and \cite{Barkley91, Biktashev03, Postnikov16} for the FitzHugh-Nagumo model. In section 2 we show under what assumptions the 1993 \cite{Karma93} and 1994 \cite{Karma94} versions of the Karma model are equivalent. In section 3 we present a systematic comparison of the FitzHugh model eqref{eq:FHN} and the Karma model as defined in section 2. We conclude our comparison in section 4 with numerical simulations of the full PDE systems with a focus on the Karma model.
    
\section{The Karma model}
    As mentioned above, the Karma model introduced in \cite{Karma93} in 1993 is a two-variable model involving one fast and one slow variable similarly to the FitzHugh-Nagumo model but incorporating additionally important dynamic features of the Noble model for
    cardiac cells. The Karma model equations read
	\begin{equation}\label{eq:Karma93}
		\begin{aligned}
			\varepsilon\frac{\partial E}{\partial t} &= \varepsilon^2\Delta E-E+\left(E^*-\left(\frac{n}{n_B}\right)^M\right)(1-\tanh(E-3))\frac{E^2}{2},\\
			\frac{\partial n}{\partial t} &= \theta(E-1)-n.
		\end{aligned}
	\end{equation}
	with the function $\theta(x)=\max\{x,0\}$, which is also known as a rectifier and is commonly employed currently as a rectified linear unit (ReLU) in machine learning. The parameter $0<n_B<1$ controls the position of the excitable wave and the parameter $M\gg1$ controls the insensitivity of the excitable wave velocity with respect to the slow gating variable $n$. Furthermore the constant $E^*=1.5415$ has been fitted such that
	\begin{equation}\label{eq:conditionE^*}
	    f_E(E,n_B)=\frac{\partial}{\partial E}f_E(E,n_B)=0
	\end{equation}
	for some $E$ where $f_E$ is the right-hand side of the first equation without the diffusion term.\\
	
	In a follow-up paper~\cite{Karma94} in 1994 the model was formalized in a slightly more general way as follows\\
	\begin{equation}\label{eq:Karma94}
		\begin{aligned}
			\frac{\partial E}{\partial t} &= \gamma\Delta E+\tau_E^{-1}[-E+\left(E^*-\mathscr{D}(n)\right)h(E)],\\
			\frac{\partial n}{\partial t} &= \tau_n^{-1}[\mathscr{R}(n)\theta(E-1)-(1-\theta(E-1))n].
		\end{aligned}
	\end{equation}
	One may view $\tau_E$ and $\tau_n$ as defining the scales for the reaction terms at which $E$ and $n$ respectively evolve. We therefore define $\varepsilon=\tau_E/\tau_n$ as a single parameter separating the time scales. Next, to make sure there are exactly two stable equilibria for $n$ fixed (corresponding to the depolarized and polarized states) a common choice~\cite{Karma93,Karma94} for the reaction function $h$ is
	\begin{equation}\label{eq:h}
		h(E)=(1-\tanh(E-3))\frac{E^2}{2}
	\end{equation}
	and the parameter $E^*$ is kept as defined above. Alternatively, a common suggestion~\cite{Karma93,Karma94} is a function of the form $h(E)=E^2-\delta E^3$ which we will treat in more detail later.\\
	
	To fully define the model we still have to specify the restitution function $\mathscr{R}(n)$ and dispersion function $\mathscr{D}(n)$. The former is responsible for the length $A$ of an action potential after a diastolic or rest interval of length $D$. The latter function defines the relation between the dispersion velocity $c$ of a pulse with respect to the previous diastolic interval. In theory, both functions can be chosen to fit arbitrary restitution and dispersion curves of the system to be modelled.\\
	\\
	For $\varepsilon$ small Karma presents the unique relation
	\begin{equation}
		\mathscr{R}(n)=\frac{n}{\left(\frac{\text dA}{\text dD}\right)_{D=\tau_n\ln(1/n)}}
	\end{equation}
	where $A(D)$ is the restitution curve. Choosing 
	\begin{equation}
		\mathscr{R}(n)=\frac{1-(1-e^{-Re})n}{1-e^{-Re}}
	\end{equation}
	leads to the restitution curve $A(D)=A_{max}+\tau_n\ln(1-(1-e^{-Re})e^{-D/\tau_n})$ and the control parameter $Re$ for the restitution properties.\\
	\\
	Similarly, for the dispersion curve with $\varepsilon$ small we have the relation
	\begin{equation}
		c(D)=\left(\frac{\gamma}{\tau_E}\right)^{1/2}C(\mathscr{D}(e^{-D/\tau_n}))
	\end{equation}
	where $c(D)$ is the dispersion curve and $C$ is a function that can be fitted numerically by a third order polynomial. Karma chooses the simple dispersion function
	\begin{equation}
		\mathscr{D}(n)=n^M
	\end{equation}
	with the control parameter $M$ for the dispersion properties.\\
	
    Having the full definition of the model of 1994 we have to check that both versions of the model, introduced in 1993 and 1994 respectively, are in fact equivalent.
	
	\begin{proposition}
	    The models \eqref{eq:Karma93} and \eqref{eq:Karma94} with the functions $h(E),\mathcal{R}(n)$ and $\mathcal{D}(n)$ chosen as above are equivalent for an appropriate value of $\gamma$.
	\end{proposition}
	\begin{proof}
    	We start by rescaling time in equations \eqref{eq:Karma94} with $\tilde t=\tau_n^{-1}t$ and dropping the tildes
    	\begin{equation}
    		\begin{aligned}
    			\varepsilon\frac{\partial E}{\partial t} &= \tau_E\gamma\Delta E+-E+\left(E^*-\mathscr{D}(n)\right)h(E)\\
    			\frac{\partial n}{\partial t} &= \mathscr{R}(n)\theta(E-1)-(1-\theta(E-1))n
    		\end{aligned}
    	\end{equation}
    	Furthermore we rescale the gating variable $\tilde n=n/n_B$ and use the parameter transformation $n_B=1-e^{-Re}$. Again, after dropping the tildes we have
    	\begin{equation}
    		\begin{aligned}
    			\varepsilon\frac{\partial E}{\partial t} &= \tau_E\gamma\Delta E-E+\left(E^*-\left(\frac{n}{n_B}\right)^M\right)h(E)\\
    			\frac{\partial n}{\partial t} &= \theta(E-1)-n
    		\end{aligned}
    	\end{equation}
    	which differs from the 1993 model \eqref{eq:Karma93} only in the diffusion parameter. Since the model is non-dimensional we can assume without loss of generality that $\tau_E=1$. In particular $\tau_E$ is independent of $\varepsilon$ so we have again a slightly more general formulation of the same model. By choosing $\gamma=\varepsilon^2$ we obtain the 1993 model.
	\end{proof}
	
	
	In the remainder of this paper we use the simpler form of the reaction function mentioned above. To stay as close to the function used by Karma as possible we have chosen the reaction function
	\begin{equation}
	    h(E)=2(E^2-\frac{1}{4}E^3)
	\end{equation}
	which is the third-order Taylor expansion of \eqref{eq:h} at $E=3$. 
	\begin{center}
		\includegraphics[width=0.75\textwidth]{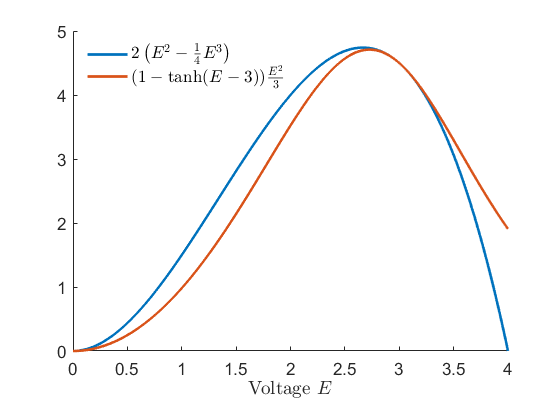}
		\captionof{figure}{Reaction function used in \cite{Karma93, Karma94} (orange) together with the function $h(E)$ we are going to analyse in this paper (blue).}
	\end{center}
	Due to the change in the reaction function we have to adapt the parameter $E^*$ such that condition \eqref{eq:conditionE^*} is still satisfied, which yields $E^*=1.5$.\\
	
	Summarizing, in the reminder of the paper we are analysing the following model equations
    \begin{equation}\label{eq:Karma}
        \begin{aligned}
        	\frac{\partial E}{\partial t} &= D\Delta E-E+2\left(E^*-n^M\right)(E^2-\delta E^3)\\
        	\frac{\partial n}{\partial t} &= \varepsilon\left(\frac{1}{n_B}\theta(E-1)-n\right)
        \end{aligned}
    \end{equation}
    for $E^*=1.5$ and $\delta=0.25$ where we used a mixed form of the scalings in \cite{Karma93} and \cite{Karma94}.
		
\section{FitzHugh-Nagumo and Karma model}

    In this section we proceed analysing and comparing the Karma model \eqref{eq:Karma} to the classical FHN system \eqref{eq:FHN}. We note that both models are two-dimensional systems with a clear fast-slow structure and a diffusive term for the fast variable representing the voltage. A concise introduction into the mathematical theory we will be applying is given in Appendix \ref{sec:theory}. In this paper, we are going to focus on illustrating and extracting the main geometric and analytical insights needed in the proofs of different types of dynamics, which makes it more transparent, where the similarities and differences between the two models are, and how to interpret these differences biologically. 
    
    
    \subsection[Pure ODE models]{Pure ordinary differential equations (ODE) models}\label{sec:ODE}
        We start by comparing the simplified version of both models by considering the pure ODE models, i.e. we set the diffusion coefficients equal 0. Hence we are working with the equations
        \begin{equation}\label{eq:KarmaODE}
        	\begin{aligned}
        		E' &=-E+2\left(E^*-n^M\right)(E^2-\delta E^3)+I\\
        		n' &= \varepsilon\left(\frac{1}{n_B}\theta(E-1)-n\right)
        	\end{aligned}
        \end{equation}
        with $E^*=1.5$, $\delta=0.25$, $M\gg1$ and $0<n_B<1$ for the Karma model and comparing them to the FHN equations
        \begin{equation}\label{eq:FHNODE}
        	\begin{aligned}
        		v'&=v-\frac{v^3}{3}-w+I\\
        		w'&=\varepsilon(v+a-bw)
        	\end{aligned}
        \end{equation}
        with $0<b<1$ and $1-\frac{2}{3}b<a<1$.

        \subsubsection{FitzHugh-Nagumo}\label{sec:FHN}
        
        	The FitzHugh-Nagumo model has been analyzed extensively in the literature due to its simplicity and generality. In this paper we choose the parameter values $a=0.7$ and $b=0.8$ as standard configuration for cardiac cells following \cite{FitzHugh61} such that for $I=0$ the unique equilibrium is stable corresponding to the polarized state. For completeness we present now a short overview over the most important steps of the analysis of the ODE system when $I=0$ by exploited the time-scale separation in the system. For more extensive proofs and deeper analysis of the FitzHugh-Nagumo model see \cite{Jones91, Jones95, Rauch78, Rocsoreanu00}.\\
        	 
        	In the FitzHugh-Nagumo model \eqref{eq:FHNODE} the flow is always controlled by the third order critical manifold as we can observe in its phase plane in Figure \ref{fig:FHN}. The manifold can be divided by its extrema into three branches, where the outer ones are attracting and the middle branch is repelling, therefore the flow away from the critical manifold will approach fast to one of the outer branches. When $I=0$, orbits on the middle or close to the right branch follow the slow flow upwards towards the maximum where they jump fast towards the left branch. Once close to the left branch every orbit will finally converge towards the sable equilibrium.\\
        	
        	\begin{center}
            	\includegraphics[width=.7\textwidth]{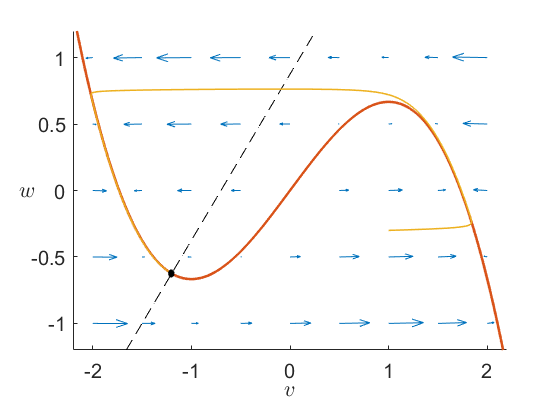}
            	\captionof{figure}{\label{fig:FHN}Phase plane of the FitzHugh-Nagumo system for $I=0$. We can see the critical manifold (orange) and the $w$-nullcline (dashed) as well as a prototypical orbit (yellow) converging to the global equilibrium (black).}
        	\end{center}
        	
        	The next theorems formalize the main results. The proof is based on the decomposition of the different time scales when $\varepsilon$ is small. For this reason we first look in Theorem \ref{th:FHN0} at the singular limit separating the analyse of the layer problem and the reduced system before constructing the candidate orbits. Finally in Theorem \ref{th:FHNepsilon} we perturbed the candidate orbits for $\varepsilon>0$.
        	
        	\begin{theorem}\label{th:FHN0}
            	In the singular limit $\varepsilon=0$ of the FitzHugh-Nagumo equations \eqref{eq:FHNODE} with $I=0$ we have a unique stable equilibrium and the third order critical manifold
            	\begin{equation}\label{eq:C0FHN}
            		C_0=\{(v,w):w=v-\frac{v^3}{3}+I\}.
            	\end{equation}
            	Every candidate orbit can be constructed as concatenation of fast segments converging to one of the outer branches of $C_0$ which are attracting and slow segments on the critical manifold. Eventually all orbits converge to the fixed point.
        	\end{theorem}
        	\begin{proof}[Sketch of the Proof.]
            	The layer problem defines a one-dimensional system with equilibria given by \eqref{eq:C0FHN} and their Jacobian is negative on the outer branches, positive in the middle branch and 0 on both extrema. The non-zero eigenvalues divide the critical manifold into 2 attracting outer branches and a repelling inner branch while both non-hyperbolic points satisfy the conditions of generic folds.\\
            	On the other hand, the slow flow is determined by the $w$-nullcline which is a linear function. On the left of this line the flow moves downwards while it moves upwards on the right. Last, the $w$-nullcline crosses $C_0$ exactly once, for $I=0$ the intersection occurs on the left branch of the critical manifold which results in a stable equilibrium as previously mentioned.\\
            	The candidate orbits can now be constructed following first the fast flow towards one of the attracting branches. On the middle and right branch of the critical manifold we converge by the slow flow to the maximum where we switch again to a fast fiber connecting to the left branch. Finally, on the left branch the slow flow converges to the equilibrium.
        	\end{proof}
        	
        	To finish our analysis we show that the candidate orbits we constructed when $\varepsilon=0$ correspond in fact to solutions of the FitzHugh-Nagumo model for $\varepsilon>0$.
        	
        	\begin{theorem}\label{th:FHNepsilon}
        	    The candidate orbits found in the singular limit $\varepsilon=0$ of equations \eqref{eq:FHNODE} when $I=0$ can be perturbed to solution curves of the full system with $\varepsilon>0$.
        	\end{theorem}
        	\begin{proof}[Sketch of the Proof.]
            	We have already seen that, excluding both extrema, that the critical manifold is normally hyperbolic. Therefore, choosing an appropriate compact subset of $C_0$, all the conditions of Fenichel's theorems \ref{th:fenichel1}-\ref{th:fenichel3} are satisfied, which means that the the slow flow on the critical manifold and switches from the fast fibers to the slow flow persist under a smooth perturbation. Finally, we need to transform a neighbourhood of the extrema into the normal form of a generic fold point to apply geometric blow-up as introduced in \cite{Krupa01a} (see also Appendix \ref{sec:blowup}). This method provides the persistence of the switches at the maximum and minimum.
            \end{proof}
    
        \subsubsection{Karma: No external current}
        	
        	To understand the Karma model equations \eqref{eq:KarmaODE} we will now perform a similar analysis. We will show that the dynamics of the 
        	Karma model are similar to FitzHugh-Nagumo since again the system is controlled by the critical manifold presenting a similar shape as shown in Figure \ref{fig:Karma}. As before, we shall indicate the main geometric steps of each proof; see also the appendix for more background on the geometric view via geometric singular perturbation theory.\\
        	
        	As before we have two attracting branches separated by a repelling one and exactly one stable equilibrium. An arbitrary orbit will either approach the right branch and then slowly ascend towards the fold point, where it jumps towards the left branch or it approaches directly the left branch where it slowly converges to the stable equilibrium at the origin. In contrast to FHN, the Karma model shows in addition to the stable equilibrium two further unstable fixed points. In general, these points do not affect the overall dynamics, however, the system \eqref{eq:KarmaODE} has not only two additional fixed points, which do not converge to the stable equilibrium but also a slow singular heteroclinic orbit between them. 
        	
        	\begin{center}
        		\includegraphics[width=0.7\textwidth]{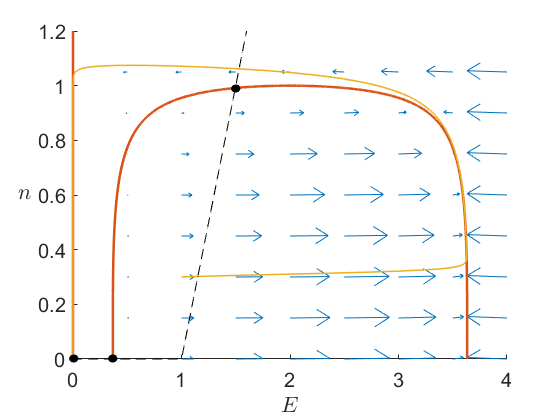}
        		\captionof{figure}{Phase plane of the Karma system for $M=4,$ $\varepsilon=10^{-2}$ and $I=0$. We can see the critical manifold (orange) and the $w$-nullcline (dashed) as well as the global equilibria (black) and a prototypical orbit (yellow) converging to the stable fixed point (0,0).}\label{fig:Karma}
        	\end{center}
        	
        	To formally analyse the dynamics, we want to exploit the different time scales as we did for FitzHugh-Nagumo and therefore consider first the limit $\varepsilon=0$ analysing the layer and the reduced problem separately.
        	
        	\begin{remark}
        	    In contrast to FHN the Karma model is continuous but not smooth due to the rectifier in the second equation. Although some of the analysis techniques used for FitzHugh-Nagumo have to be modified or extended, the existence and uniqueness of solutions is still guaranteed by the Picard-Lindel\"off Theorem.
        	\end{remark}
        	
        	\begin{theorem}\label{th:Karma0}
            	In the singular limit $\varepsilon=0$ of the Karma equations \eqref{eq:KarmaODE} with $I=0$ we have a stable, an unstable and a saddle equilibrium and the critical manifold
            	\begin{equation}\label{eq:C0Karma}
            		C_0=\left\{(E,n): E=0 \text{ or } n=\sqrt[M]{E^*-\frac{1}{2E(1-\delta E)}}\right\}.
            	\end{equation}
            	Every candidate orbit can be constructed as concatenation of fast segments converging to one of the outer branches of $C_0$, which are attracting and slow segments on the critical manifold. An orbit will either eventually converge to the stable equilibrium at the origin or it is one of the two unstable fixed points or a unique heteroclinic orbit between them.
        	\end{theorem}
        	\begin{proof}[Sketch of the Proof.]
            	\textbf{Layer problem:} Like in FHN we have the one-dimensional fast subsystem
            	\begin{equation}\label{eq:Karmafast}
            	    E'=-E+2(E^*-n^M)(E^2-\delta E^3)
            	\end{equation}
            	where $n$ is a parameter. In this subsystem $E=0$ is always an equilibrium and, depending on $n$ we have either two further equilibrium points for $n<1$, exactly one for $n=1$, or no further equilibria when $n>1$. We can calculate the derivative on these points and get
            	\begin{equation}\label{eq:JacobianKarma}
            	    J(E;n) =-1+2(E^*-n^M)(2E-3\delta E^2),
            	\end{equation}
            	which for $(E,n)\in C_0$ is negative on the outer branches, positive in the middle branch and 0 only at 
            	$$p_F=(2,1).$$ 
            	Therefore the critical manifold is normally hyperbolic everywhere except at $p_F$. It is straightforward to check that $p_F$ satisfies all the conditions of a generic fold point.\\
            	
            	\textbf{Reduced problem:} The slow flow has a piece-wise linear nullcline  which intersects $C_0$ exactly three times as shown in Figure \ref{fig:Karma}: twice on the $E$-axis and once with $E,n>0$. Therefore, we have three global fixed points: a stable equilibrium at the origin, a saddle at the intersection of the unstable branch of $C_0$ and the $E$-axis and an unstable fixed point on the unstable branch. The reduced problem is given by
            	\begin{equation}\label{eq:Karmaslow}
            	    \dot n=\frac{1}{n_B}\theta(E-1)-n~,~~~(E,n)\in C_0
            	\end{equation}
            	this means that the slow flow on the left branch as well as on the middle branch below the unstable equilibrium points downwards while it points upwards on the right branch as well as between the unstable node and the fold point $p_F$.\medskip
            	
            	Combining this information we now want to construct the candidate orbits in the singular limit. Any orbit starting away from the critical manifold will first follow the fast flow converging to one of the attracting branches of $C_0$. The orbits on the right branch follow then the slow flow upwards to $p_F$ where they jump with the fast flow to the left branch. There, all orbits follow the slow flow downwards converging to the global equilibrium $(0,0)$. Orbits starting on the unstable branch of the critical manifold will either converge to $p_F$ and jump to the left branch if they start above the unstable node or they will converge downwards towards the saddle if they start below the unstable fixed point.
        	\end{proof}
        	
        	Finally we show in the following two theorems that the Karma model for $\varepsilon>0$ has an equivalent behaviour as in the singular limit. To prove this we want to apply geometric singular perturbation theory (see Appendix \ref{sec:fenichel} for more details). Nevertheless this theory requires differentiability of the system which we loose when $E=1$. Since this line crosses the repelling branch of the critical manifold below the unstable node we excluded the heteroclinic connection between the unstable fixed points in Theorem \ref{th:Karma1}. This segment will be analyzed separately in Theorem \ref{th:Karma2}.\\
        	
        	\begin{theorem}\label{th:Karma1}
        	    Away from the heteroclinic segment of $C_0$ between the saddle and the unstable node, candidate orbits found in the singular limit $\varepsilon=0$ of equations \eqref{eq:KarmaODE} with $I=0$ can be perturbed to solution curves of the full system with $\varepsilon>0$.
        	\end{theorem}
        	\begin{proof}[Sketch of the Proof.]  
        	    For this proof we first need to divide our phase space along the line $E=1$ to be able to guarantee smoothness, therefore we will analyse the left and right parts of the critical manifold separately. Furthermore, Fenichel's theorems require smooth vector fields defined on $\mathbb{R}^2$. In order to satisfy this condition we extend the systems on each side to the entire real plane so that we will be working with either
        	    \begin{equation}
                	n' = -\varepsilon n ~~~~\text{or}~~~~ n'=\varepsilon\left(\frac{1}{n_B}(E-1)-n\right)
                \end{equation}
                and the unchanged fast equation \eqref{eq:Karmafast} defined in both cases for $(E,n)\in\mathbb{R}^2$. Since all the results of Fenichel's Theory are local around the subset of the critical manifold we are focusing on, these extensions do not change the results.\\
            	Away from the fold point $p_F$ we determined above that the critical manifold is normally hyperbolic so, taking any compact subset of the left branch, we are able to apply Fenichel's first Theorem \ref{th:fenichel1} to perturb the attracting and repelling branches separately to locally invariant manifolds of the full system. Furthermore, by Fenichel's second and third Theorem \ref{th:fenichel2} and \ref{th:fenichel3} the switching between the fast and the slow flow is also preserved for $\varepsilon>0$. Returning now to our original system \eqref{eq:KarmaODE} we can extend the fast fibres over $E=1$ using the continuity of the flow. Last it remains to prove that the switching at $p_F$ is preserved as well. We know that this point is a generic fold so we can do a coordinate transformation to normal form. Krupa and Szymolyan presented in detail in \cite{Krupa01a} the analysis of the normal form of a generic fold by performing a geometric blow-up (see also Appendix \ref{sec:blowup}) which applied to our model concludes the proof. 
        	\end{proof}
        	    
        	\begin{theorem}\label{th:Karma2}
        	    The heteroclinic segment of the critical manifold in the Karma model \eqref{eq:KarmaODE} with $I=0$ can be perturbed to a heteroclinic orbit between the equilibria for $\varepsilon>0$.
        	\end{theorem}
        	\begin{proof}[Sketch of the Proof.]
        	    Following the proof of Theorem \ref{th:Karma1} we are able to perturb any compact subset of the heteroclinic segment without the point at $E=1$ but we cannot directly guarantee that the left and right subsets connect. To demonstrate the existence of the expected heteroclinic orbit connecting the unstable equilibrium to the saddle we need to look directly at the system with $\varepsilon>0$.\\
        	    Since both unstable manifolds of the saddle converge by the analysis above to the origin we are able to define an invariant set delimited by them as shown in Figure \ref{fig:heteroclinic(E,n)}. Furthermore, from the previous analysis we know that every orbit starting away from the heteroclinic segment will eventually converge to the origin so we follow that there is no periodic orbit and therefore no limit cycle in this set. Now we are able to apply the Poincar\'e-Bendixson Theorem in the limit $t\to-\infty$. Since the origin is unstable in backward time and there are no limit cycles the theorem shows that in fact the now unstable manifold of the saddle needs to converge to the now stable equilibrium proving the existence of the expected heteroclinic orbit for positive $\varepsilon$.
    	    \end{proof}
    	    \begin{center}
    	        \includegraphics[width=0.7\textwidth]{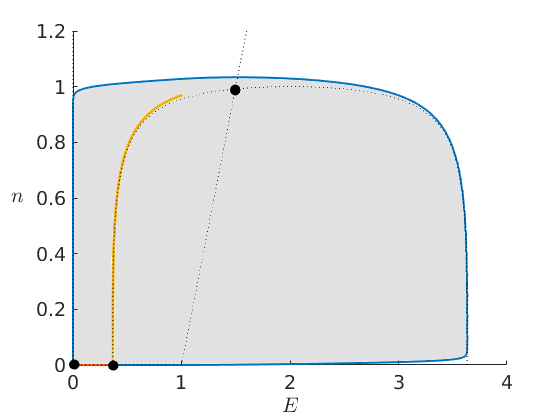}
    	        \captionof{figure}{Phase plane $(E,n)$ for $\varepsilon>0$. In grey we have the invariant set enclosed by the unstable manifolds of the saddle (blue and orange). Furthermore we have its stable manifold for $E<1$ (yellow), the nullclines (dotted) and the equilibria of the system (black)}\label{fig:heteroclinic(E,n)}
    	    \end{center}
        
        In summary, we have shown that although the general structure of the Karma ODE model is similar to the FHN ODE model, there are subtle mathematical differences, particularly in the case of using the standard variants in the literature.
        	
        \subsubsection{Karma: External current $I>0$} \label{sec:KarmaI}
        Next we focus on the case where the external current $I>0$. In the FitzHugh-Nagumo model it is well known that an external current shifts the critical manifold upwards as shown in Figure \ref{fig:FHNI}. Without changing anything else, the dynamics switch from a stable resting state to an oscillatory behaviour to a stable depolarized state as the input $I$ increases. To mathematically show this different behaviours we can perform an analogous, yet more complicated, analysis as presented in Section \ref{sec:FHN}; see also~\cite{Krupa01b}.\\
        	\includegraphics[width=0.5\textwidth]{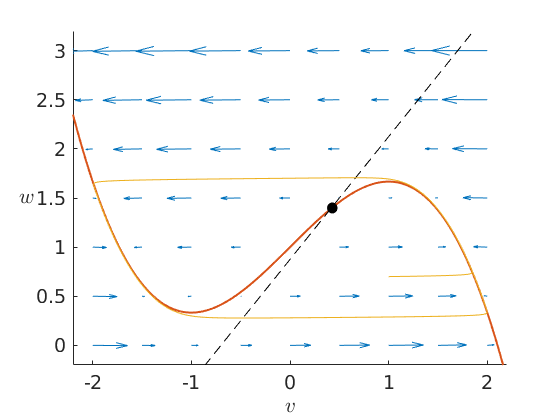}
        	\includegraphics[width=0.5\textwidth]{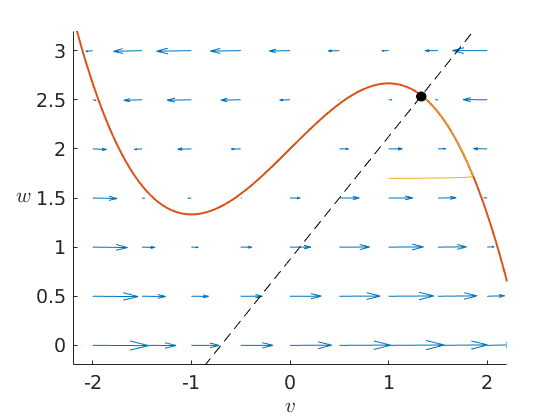}
        	\captionof{figure}{Phase plane of the FitzHugh-Nagumo system for $I=1$ (left) and $I=2$ (right). We can see the critical manifold (orange) and the $w$-nullcline (dashed) as well as the global equilibrium (black) and a prototypical orbit (yellow) oscillating for $I=1$ and converging to the stable equilibrium for $I=2$.\\}\label{fig:FHNI}
        	
        	In contrast to that, adding an external current to the Karma model results in a significant change in the shape of the critical manifold as shown in Figure \ref{fig:KarmaI}. While we have a regime of $I$ where the critical manifold is ``S''-shaped comparable to FHN giving rise to similar relaxation oscillations, when $I$ is big, the manifold flattens out in such a way that the curve is monotonous. In particular this means that the model does not allow any relaxation oscillations or pulses for a high input $I$. Furthermore, in the Karma model the stable resting state disappears when it collides with the saddle in a fold bifurcation whereas in FHN only the stability of the already unique equilibrium changes. Lastly, the change of stability of the unstable node is for the most part independent of the shape of $C_0$. This means that, depending on the model parameters, we can observe bistability as well as a relaxation pulse with a stable depolarized state similar to FHN in addition to the dynamics we have already described.\\
        	\includegraphics[width=0.5\textwidth]{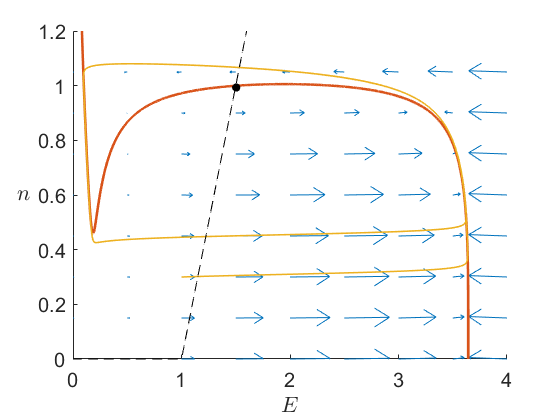}
        	\includegraphics[width=0.5\textwidth]{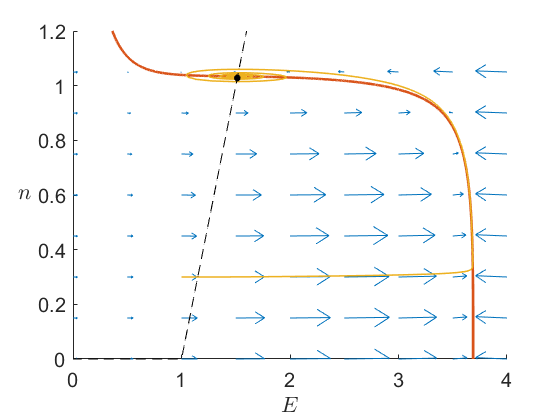}
        	\captionof{figure}{Phase plane of the Karma system for $M=4$, $\varepsilon=10^{-2}$ and $I=0.1$ (left) and $I=0.5$ (right). We can see the critical manifold (orange) and the $w$-nullcline (dashed) as well as the global equilibria (black) and a prototypical orbit (yellow) oscillating for $I=0.1$ and converging to the unique stable equilibrium for $I=0.5$.\\}\label{fig:KarmaI}
        	
        	The next theorem formalises all the different dynamic regimes described above.
        	
        	\begin{theorem}\label{th:KarmaI}
            	In the singular limit $\varepsilon=0$ of the Karma equations \eqref{eq:KarmaODE} with $I>0$ we have the critical manifold
            	\begin{equation}\label{eq:C0KarmaI}
            	    C_0=\left\{(E,n): n=\sqrt[M]{E^*-\frac{E-I}{2E^2(1-\delta E)}}\right\}.
            	\end{equation}
            	Every candidate orbit can be constructed as concatenation of fast segments converging to one of the outer branches of $C_0$ which are attracting and slow segments on the critical manifold but we need to differentiate multiple parameter regimes for $I$ giving rise to different overall dynamics.\\
            	The first threshold is given by $I=I_2$ where the equilibrium with $n>0$ changes from unstable to stable as $I$ increases. Furthermore we have
            	\begin{itemize}
            	    \item $I<I_0\approx 0.08718$: The equations have three equilibrium points with a stable node and a saddle on the $E$-axis. If $I<I_2$ the behaviour is equivalent to the case $I=0$. Otherwise, the system is bistable.
            	    \item $I_0<I<I_1=\frac{4}{9}$: The equations have a unique equilibrium point. If $I<I_2$ the equilibrium is unstable and the system has a stable relaxation oscillation. Otherwise, the equilibrium is globally stable although there are relaxation pulses.
            	    \item $I>I_1(>I_2)$: The middle branch of $C_0$ disappears such that the critical manifold is attracting everywhere and the unique equilibrium is globally stable.
            	\end{itemize}
        	\end{theorem}
        	\begin{proof}[Sketch of the Proof.]
            	Analogously to the previous section, we are now going to study separately the fast and slow subsystems in the singular limit in order to proof Theorem \ref{th:KarmaI}.\medskip
            	
            	\textbf{Layer problem:} The layer problem is defined by the equation
                \begin{equation}
                    E'=-E+2(E^*-n^M)(E^2-\delta E^3)+I
                \end{equation}
                for $n$ fixed. We can see directly that the derivative of the right-hand-side is still given by \eqref{eq:JacobianKarma}. Since by definition any equilibrium is contained in $C_0$ we can plug in the equality 
                \begin{equation*}
                    (E^*-n^M)=\frac{E-I}{2E^2(1-\delta E)}
                \end{equation*}
                and rewrite that way the Jacobian depending on the external current $I$ instead of $n$ as follows
                \begin{equation}
                    J(E;I)=-1+\frac{(E-I)(2-3\delta E)}{E(1-\delta E)}.
                \end{equation}
                To isolate any non-hyperbolic equilibrium of the fast system we set the $J(E;I)=0$ and obtain after simplifying 
                \begin{equation}
                    0=2E^2-(\frac{1}{\delta}+3I)E+\frac{2}{\delta}I.
                \end{equation}
                Solving the quadratic equation for $E$ we find 2 curves of non-hyperbolic equilibria given by
                \begin{equation}
                    E_\pm(I)=\frac{(4+3I)\pm\sqrt{(4+3I)^2-64I}}{4}
                \end{equation}
                which connect and disappear for $I\geq I_1:= \frac{4}{9}$. It is important to check for which values of $I$ the curves $E_\pm(I)$ are in fact on the critical manifold, more precisely, whether 
                \begin{equation}\label{eq:ineqKarmaI}
                    E^*-\frac{E_\pm-I}{2E_\pm^2(1-\delta E_\pm)}\geq0.
                \end{equation}
                The curve $E_+(I)$ satisfies this inequality for all $I\in [0,I_1]$ but for the curve $E_-(I)$ the inequality \eqref{eq:ineqKarmaI} is only satisfied when 
                $$I\in[I_0,I_1]$$
                with $I_0 \approx 0.08718$.\medskip
            	
        	    Having isolated the non-hyperbolic equilibria we check that, similar to the previous section, when $I<I_1$ we have a division of the critical manifold into three separate branches where the Jacobian is negative on the outer ones and positive in the middle branch. When $I>I_1$ the Jacobian stays negative along the whole critical manifold.\\
            	
            	\textbf{Reduced problem:} The slow subsystem is still defined by \eqref{eq:Karmaslow} but now we have a different definition of $C_0$. The most important change lies in the fact that the $n$-nullcline will, due to continuity, cross the curve $E_\pm(I)$ in the $(E,n)$-plane for some $I_2\leq I_1$ dependent on the system parameters $n_B$ and $M$ as we increase $I$. By crossing this curve the global equilibrium of the system (with $n>0$) changes its stability and becomes stable. Furthermore we have already seen that the two equilibria at the $E$-axis collide and disappear for $I=I_0$ so that for $I>I_0$ we only have 1 equilibrium of the slow flow.\\
                
                \begin{remark} 
                    Looking at the full system we identify $I=I_0$ as the bifurcation parameter where the system undergoes a saddle-node bifurcation when the 2 equilibria on the $E$-axis collide and disappear giving rise to the curve $E_-(I)$. For the corresponding values of $I$ we can check again that the conditions for a generic fold point are satisfied on both curves $E_\pm(I)$ everywhere except for the point $I=I_1$ and the singularity at $E_+(I_2)$ or $E_-(I_2)$. At the first one the system undergoes a cusp bifurcation where the two fold points annihilate each other. We will come back to this bifurcation later on in more detail. Last, the intersection between the $n$-nullcline and $E_\pm(I)$ at $I=I_2$ satisfies the conditions of a nondegenerate fold but the slow flow is 0. We conclude that at this point we have a fold singularity.\\
                \end{remark}
            	
            	Finally we construct the candidate orbits in the singular limit in the different parameter regimes.
            	\begin{itemize}
            	    \item $I<I_0$: If $I<I_2$ the orbits are equivalent to without incoming current.\\
            	    If $I>I_2$ then the fast flow will converge to one of the attracting branches of $C_0$ but while every orbit on the left branch still converges to the origin, contrary to the previous case, all orbits on the right branch will stay on that branch converging to the second stable equilibrium. The slow flow on the repelling branch converges like before to the saddle point.
            	    \item $I_0<I<I_1$: First every orbit follows the fast fibres to one of the attracting branches of the critical manifold.\\
            	    If $I<I_2$ the slow flow leads then to the next fold point where we can again use a fast fibre to jump to the other attracting branch forming a cycle. The flow on the repelling branch will converge away from the equilibrium to the folds following from there the cycle.\\
            	    If $I>I_2$ and assuming the $n$-nullcline crosses $E_+(I)$ then the flow on the left and middle branch will still converge to the minimum jumping to the right branch. There all orbits converge to the equilibrium. The case where the $n$-nullcline intersects $E_-(I)$ is equivalent subject to interchange left and right and taking the maximum instead of minimum.
            	    \item $I_1<I$: The entire critical manifold is attracting so every orbit flows fast to it and then converges to the unique equilibrium.
            	\end{itemize}
            \end{proof}
            
            \begin{remark}
                To go briefly into the biophysical implications of the above observations we note that all $I_0,I_1,I_2$ are important thresholds affecting differently the behaviour of the cell. Whenever we have a background stimulation $I>I_2$, any cell which depolarises over this threshold would not be able to repolarise anymore and will therefore cease to ``fire'' further signals. On the other hand, a background stimulation $I_0<I<I_1$ and $I<I_2$ results in a self-excitatory system which will ``fire'' regularly. Finally, when the background stimulation is higher than $I_1$ the cell will automatically depolarise so that any future signal is blocked.
            \end{remark}
        	
        	\begin{theorem}
        	    Whenever $E_-(I)\neq1$, candidate orbits found in the singular limit $\varepsilon=0$ of equations \eqref{eq:KarmaODE} with $I>0$ away form the bifurcation points $I_0$ and $I_2$ can be perturbed to solution curves of the full system with $\varepsilon>0$.
        	\end{theorem}
        	\begin{proof}[Sketch of the Proof.]
        	    Analogously to Theorem \ref{th:Karma1} we find that also for $I>0$ away from the intersection between $E=1$ and the critical manifold we can perturb every orbit as expected for $\varepsilon>0$. In the case when $E_-(I)>1$, in particular when $I>I_1$, this point lies in the left branch of $C_0$. After continuing the slow manifold obtained for $E<1$ over this line we can use the attracting properties of the slow manifold for $E>1$ to follow that both manifolds will approach each other. Recalling that the slow manifold is not unique we can directly choose the continuation of the left part to also be our representative slow manifold for $E\geq1$. To finish the proof we need to separate the different parameter regimes when $E_-(I)<1$. If we first take $I<I_2$ we have the following cases.
        	    \begin{itemize}
        	        \item When $I<I_0$ the system is equivalent to the case with $I=0$ and the proof of Theorem \ref{th:Karma2} can still be applied to derived the heteroclinic orbit between the unstable node and the saddle point.
        	        \item When $I_0>I>I_1$ we have already derived a stable limit cycle with the unstable fixed point the only orbit not converging to it. In particular we know there are no further periodic orbits inside the limit cycle. This means that, defining an invariant set delimited by the cycle, we can use the Poincaré-Bendixson Theorem to show that the segment of repelling slow manifold with $E>1$ will converge to the limit cycle for $t\to\infty$ as well as that the segment with $E<1$ will converge to the equilibrium for $t\to-\infty$.
        	    \end{itemize}
        	    Finally we look at the system with $I>I_2$. By reversing time the repelling branch of the critical manifold becomes attracting and so we can use the same technique applied above and choose the continuation of the left segment of the middle branch as slow manifold. Following the analysis given by Fenichel's theorems and geometric blowup we follow that the middle branch of the slow manifold flows over the fold point diverging in backward time. In the case where $I<I_0$ this manifold defines a separatrix dividing the phase space into the basins of attraction of the two stable equilibria. In the case where $I>I_0$ this slow manifold separates the orbits converging directly to the stable equilibria and the orbits which perform a relaxation pulse over the left or right branch of $C_0$ and one of the fold points before converging.
        	\end{proof}
        	    
    	   The limit case with $E_-(I)=1$ cannot be analysed with the methods used above since the geometric blowup also requires higher regularity. By continuity we would expect that we can still perturb the candidate orbits for $\varepsilon>0$ but this still has to be proven rigorously. Furthermore, when $I=I_0$ or $I=I_2$ the system has folded singularities. It is known that in small neighbourhoods around this points we can find canards and so-called canard explosions. For more details about this solutions see \cite{Dumortier96, Krupa01a, Krupa01b, Kuehn15}.
        	
        	\begin{remark}
        	    All the existence results obtained by Fenichel's Theory require $\varepsilon$ to be ``small enough''. In applications we need to check for every case independently what ``small enough'' means specifically.\\
        	    
        	    By looking at simulations we see that when $I<I_2$ the orbits behave as expected even for relatively large $\varepsilon\approx10^{-1}$. Nevertheless, when the equilibrium changes stability for $I=I_2$, the orbits oscillate around the equilibrium instead of converging through a slow manifold as expected from Fenichel's Theory even for very small $\varepsilon\approx10^{-4}$. This shows that even knowing that there exists an $\varepsilon$ for which this theory is applicable, for some values it is not the case. To understand what actually happens at this point with reasonable $\varepsilon$ we have to look at the eigenvalues of the fast subsystem as well as of the full system.\\
        	    
        	    Although the Jacobian $J$ of the fast subsystem is strictly smaller than 0 the critical manifold stays very close to non-hyperbolicity and so the absolute value of $J$ is very small. If we calculate the eigenvalues of the full system at the unique equilibrium we have:
            	\begin{equation*}
            	    \lambda_\pm = \frac{J-\varepsilon}{2} \pm \frac{1}{2}\sqrt{(J-\varepsilon)^2-4\varepsilon\left(-J+\frac{2}{n_B}Mn^{M-1}(E-1)(E^2-\frac{1}{4}E^3)\right)}
            	\end{equation*}
            	It holds that $J<0$ and the equilibrium is away from $E=1$ and $E=4$ therefore we see that the parenthesis in the second term of the discriminant is always positive and of order 1 w.r.t. $\varepsilon$ so the whole summand is in $\mathcal{O}(\varepsilon)$. Nevertheless, if $J\in\mathcal{O}(\varepsilon)$ then the first summand is of order $\varepsilon^2$ such that the eigenvalues become complex and our equilibrium is a stable spiral instead of a stable node.\\
            	
            	Nevertheless, the equilibrium is still globally stable and every orbit will eventually converge to it.
        	\end{remark}
        	
        	\paragraph{Extended system $(E,n,I)$.} 
        	    We have seen that an external current has an important impact on the Karma model. Therefore we next investigate an extended 3-dimensional systems with the additional slow equation
        	    \begin{equation*}
        	        I'=0
        	    \end{equation*}
        	    Figure \ref{fig:cusp} shows the critical manifolds for the Karma model as well as FitzHugh-Nagumo.\\
        	    \hspace*{-5mm}
            	\includegraphics[width=0.53\textwidth]{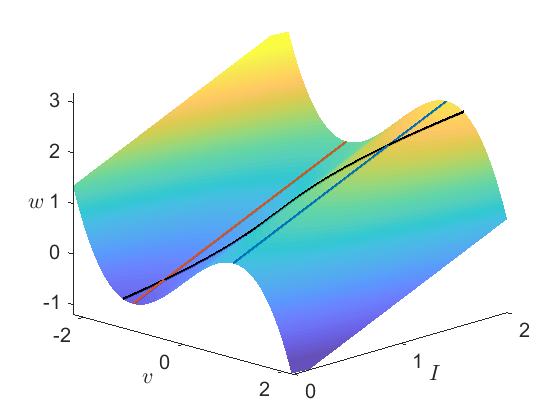}
            	\includegraphics[width=0.53\textwidth]{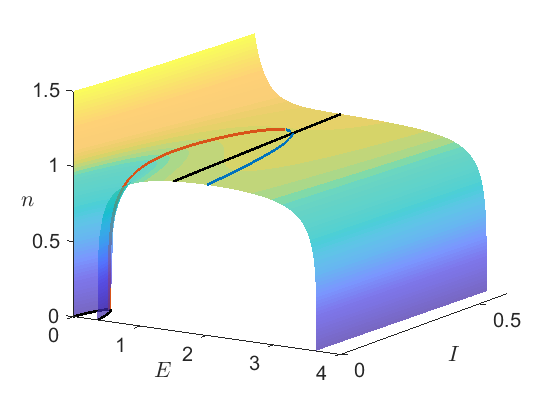}
            	\captionof{figure}{Critical manifold for FitzHugh-Nagumo (left) and Karma (right) for $M=4$ and $\varepsilon=10^{-2}$ in the $(E,n,I)$-space. In black we show the global equilibria and the red and blue curves represent the two curves of folds.\\}\label{fig:cusp}
            	
            	This representation allows for proper analysis of the fold curves. Although both models show exactly two curves of folds, their behaviour is clearly very different. Firstly we notice the extreme sensitivity of the Karma model to external currents close to $I_0$ which is not present for FHN. Furthermore, both fold curves in the FitzHugh-Nagumo model are parallel to each other while in the Karma model they collide and disappear like we had seen above. This collision point $I=I_1$, $E=\frac{4}{3}$ and the corresponding $n$ given by \eqref{eq:C0KarmaI} defines a cusp bifurcation in the Karma model. It is clearly a non-hyperbolic point and therefore it cannot be analysed using classical Fenichel theory. Nevertheless, we can do a similar analysis as for a fold point using a coordinate transformation to normal form and geometric blow-up. A detailed analysis of a cusp point using these techniques was presented by Broer et al. in \cite{Broer13}.\\
            	
            	We have shown above that both models exhibit relatively similar qualitative behaviour when considering $I$ a fixed parameter. Nevertheless, the existence of a cusp singularity presents the possibility for a diverse set of behaviours if we consider for example a slowly changing external current e.g. a slow periodic input. By extending the Karma model considering a change in $I$ we get the possibility of relaxation oscillations with a smooth return. This means that we can have oscillations whereby after a fast jump we are able to return to our starting point following only the slow dynamics. This type of behaviour is not possible in the FitzHugh-Nagumo model even after allowing changes in $I$.\\
            	
            	\begin{remark}
            	    The analysis presented in \cite{Broer13} assumes that the bifurcation point is not an equilibrium of the full system. In the Karma model this is in general the case, nevertheless we have the case where $I_1=I_2$ when the cusp point is in fact a global equilibrium. This case has (to our knowledge) not been analyzed mathematically yet and would be an interesting future extension to the current analysis. 
            	\end{remark}

    \subsection{Travelling waves}
        As the next step in our analysis we want to consider also the diffusion in the models concentrating on the existence of a travelling pulse in the 1D case. Like before, the existence as well as stability of travelling waves for the FitzHugh-Nagumo equations has been studied extensively, see for example \cite{Flores91, Hastings76, Jones84}. We are particularly interested in the construction of pulse solutions performed in \cite{Guckenheimer09}. There, Guckenheimer et al. looked at the asymmetric FHN equations
        \begin{equation}
    		\begin{aligned}
    			\frac{\partial v}{\partial t} &= D\Delta v + v(v-a)(1-v)-w+I\\
    			\frac{\partial w}{\partial t} &= \varepsilon(v-\gamma w)
    		\end{aligned}
        \end{equation}
        with the parameter values $\gamma=1$, $a=\frac{1}{10}$ and $D=5$. The system is very similar to \eqref{eq:FHN} also controlled by a cubic critical manifold in the ODE case. When we add the diffusion term the system exhibits travelling pulse solutions which they proved using a numerical continuation method for the fast fibers in the co-moving frame. In the parameter space $(w,c)$ the authors found a ``V''-shaped curve of fast heteroclinic fibers connecting the left and right branches of the critical manifold. When $c=0$ the system is Hamiltonian and there is a $w=w_*$ such that there is a double heteroclinic orbit. When $w$ is smaller than $w_*$ we have a connection from the left branch to the right one while when $w$ is bigger the connection goes in the opposite direction. A concatenation of this fibers combined with the slow flow on the critical manifold can then be perturbed analogously to the previous section for $\varepsilon>0$, although the technical details become mathematically very involved.\\
        
        Here we carry out a similar analysis for the Karma model starting by introducing the corresponding co-moving frame $z=x+ct$ such that the equations are now given by
        \begin{equation}\label{eq:KarmaPDE}
        	\begin{aligned}
        		cE_z &=DE_{zz}-E+2\left(E^*-n^M\right)(E^2-\delta E^3)+I\\
        		cn_z &= \varepsilon\left(\frac{1}{n_B}\theta(E-1)-n\right)
        	\end{aligned}
        \end{equation}
        We can easily transform the model into a first order system by introducing an additional variable $w$
        \begin{equation}\label{eq:KarmaTV}
        	\begin{aligned}
        	    E_z &= w\\
        		Dw_z &=cw+E-2\left(E^*-n^M\right)(E^2-\delta E^3)-I\\
        		cn_z &=\varepsilon \left(\frac{1}{n_B}\theta(E-1)-n\right)
        	\end{aligned}
        \end{equation}
        We now have two additional parameters with respect to the ODE model, namely $c$ and $D$. The parameter $c$ gives the velocity at which the travelling wave moves. Changing the sign of the parameter $c$ is equivalent to inverting the direction of the wave variable $z$ and substituting $w$ by $-w$. Therefore, without loss of generality we can restrict our analysis to $c>0$.\\
        
        The second parameter $D$ is the diffusion coefficient. For this parameter there are different scalings often used in the literature. Specifically in the original papers introducing the Karma model the author presents a diffusion coefficient $D\in\mathcal{O}(\varepsilon)$, introducing therefore a third scale to the system (see \cite{Karma93}) while a constant diffusion $D\in\mathcal{O}(1)$ was used in \cite{Karma94}.\\
        
        Below we focus on the model with $D\in\mathcal{O}(1)$ and for simplicity only the case without incoming current $I=0$. In the following theorems we want to illustrate that the Karma model \eqref{eq:Karma} can exhibit a travelling pulse solution with the resting state $(0,0)$ as start and end state.
        
        \begin{theorem}\label{th:KarmaTV0}
            In the singular limit $\varepsilon=0$ there exists a homoclinic candidate orbit to equations \eqref{eq:KarmaPDE} satisfying the asymptotic conditions
            \begin{equation}
                \lim_{z\to\pm\infty}(E(z),n(z))=(0,0)
            \end{equation}
        \end{theorem}
      
      We sketch the geometric idea of the proof of this result. The model, after transformation to the first order system \eqref{eq:KarmaTV}, is a $(2,1)$-fast-slow system with one-dimensional critical manifold given by
        	\begin{equation}
        	    C_0=\left\{(E,w,n): w=0,~E=0 \text{ or } n=\sqrt[M]{E^*-\frac{1}{2E(1-\delta E)}}\right\}
        	\end{equation}
        	
        	\paragraph{Reduced system.} The slow flow on $C_0$ differs from the one in the ODE model only by a factor $\frac{1}{c}$ so we are simply scaling the flow. In particular, we have the same global equilibria as before embedded into the $(E,n)$-plane.
        	
        	\paragraph{Layer problem.} The fast subsystem is defined by the equations
        	\begin{equation}
        	    \begin{aligned}
        	        E' &= w\\
        	        Dw' &= cw+E-2\left(E^*-n^M\right)(E^2-\delta E^3)
        	    \end{aligned}
        	\end{equation}
        	The equilibria correspond to the points on the critical manifold for the different values of $n$. By choosing a different representation we have the fixed point $p_0=(0,0)$ and for $n^M\leq 1.0415$
        	\begin{equation*}
        	    p_1=\left(2-2\sqrt{1-\frac{1}{2(E^*-n^M)}},0\right)~,~~p_2=\left(2+2\sqrt{1-\frac{1}{2(E^*-n^M)}},0\right)
        	\end{equation*}
        	The Jacobian at this points is given by
        	\begin{equation}
        	    J(E,w)=\begin{pmatrix}
        	    0&1\\
        	    \frac{1}{D}[1-2(E^*-n^M)(2E-3\delta E^2)]&\frac{c}{D}
        	    \end{pmatrix}
        	\end{equation}
        	with eigenvalues
        	\begin{equation*}
        	    \lambda_\pm=\frac{c}{2D}\pm\sqrt{\frac{c^2}{4D^2}+\frac{1}{D}[1-2(E^*-n^M)(2E-3\delta E^2)]}
        	\end{equation*}
        	and, when $\lambda_\pm$ are real, corresponding eigenvectors 
        	\begin{equation*}
        	    v_\pm=\begin{pmatrix}1\\\lambda_\pm\end{pmatrix}
        	\end{equation*}
        	We can directly check that the equilibria $p_0$ and $p_2$ are saddles and $p_1$ is unstable. In addition we know that $p_1$ is a node when 
        	\begin{equation*}
        	    c^2> 4D\left[2-4(E^*-n^M)\left(1-\sqrt{1-\frac{1}{E^*-n^M}}\right)\right]
        	\end{equation*}
        	and a spiral otherwise.\\
        	
        	Given the local structure around the critical manifold we want to find heteroclinic connections between $p_0$ and $p_2$ to later combine with the slow flow to heteroclinic candidate orbits.
    	
    	\begin{lemma}\label{lemma:heteroclinics}
    	    For equations \eqref{eq:Karmafast} it holds that
    	    \begin{enumerate}[(i)]
    	        \item For every $n\in[0,1]$ there exists a $c>0$ such that the system has a heteroclinic connection. When $n<\sqrt[M]{15/16}$ the orbit flows from $p_0$ to $p_2$ while for $n>\sqrt[M]{15/16}$ the orbit flows from $p_2$ to $p_0$. At $n=\sqrt[M]{15/16}$ the system has a double heteroclinic orbit in the limit $c=0$.
    	        \item For $n=1$ there exists a $c_{min}$ such that for every $c\geq c_{min}$ the system has a heteroclinic connection from $p_2$ to $p_0$.
    	    \end{enumerate}
    	\end{lemma} 
    	\begin{proof}[Sketch of proof of (i).]
        	In order to prove the first statement we are going to follow the strategy in \cite{Guckenheimer09}. Our fist step is to compute the stable and unstable manifolds of $p_0$ and $p_2$ by taking initial conditions close to the equilibria on their tangent spaces. Next, we define the plane $\Sigma$ where $E=\frac{E_2}{2}$ and calculate the intersection points $q_0$ and $q_2$ with the previously computed orbits depending on $n^M$ and $c$. The zeros of the function
        	\begin{equation}
        	    \Delta(n^M,c)=q_0(n^M,c)-q_2(n^M,c)
        	\end{equation}
        	define finally the parameters which give rise to heteroclinic orbits in the fast subsystem. Once we have one such parameter pair, the complete curve in the parameter space can be found because of continuity by slowly changing $n^M$ and computing again the zeros of $\Delta$. Figure \ref{fig:heteroclinic(n,c)} shows the computed zeros.\\
        	
            \includegraphics[width=\textwidth]{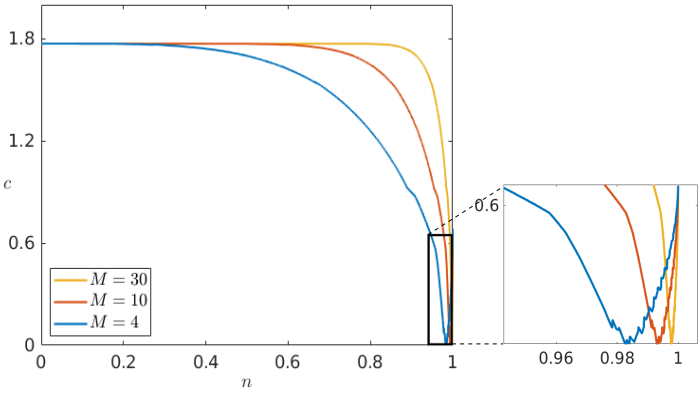}
            \captionof{figure}{Heteroclinic orbits of the fast subsystem for $D=1$ and different values of $M$ in the $(n,c)$-plane. The left branch with $n<\sqrt[M]{15/16}$ corresponds to orbits connecting the origin to $p_2$ as $z\to\infty$ while on the right branch with $n>\sqrt[M]{15/16}$ we have heteroclinic orbits connecting the equilibria in the opposite direction (see close-up). 
            \\}\label{fig:heteroclinic(n,c)}
            
            The left branch of zeros reaching from $n=0$ to $n=\sqrt[M]{15/16}$ corresponds to the intersection of the unstable manifold of $p_0$ with the stable manifold of $p_2$. The right branch (see close-up) corresponds to the unstable manifold of $p_2$ intersecting the stable manifold of $p_0$. This numerical computation could then be made rigorous, e.g., via employing rigorous numerical techniques, which are already well-established in the context of FHN~\cite{ArioliKoch}, which concludes the proof of first part of statement \textit{(i)}.
    		For the last part of the statement we observe that in the limit $c=0$ the fast subsystem is Hamiltonian with the first integral given by
    		\begin{equation}
    		    H(E,w)=\frac{1}{2}w^2-\frac{1}{2D}E^2+\frac{2}{D}(E^*-n^M)\left(\frac{1}{3}E^3-\frac{\delta}{4}E^4\right)
    		\end{equation}
    	    We calculate directly that the energy level at the origin is always 0 and $H(p_2)=0$ holds if and only if $n^M=\frac{15}{16}$. Together with the results illustrated in Figure \ref{fig:heteroclinic(n,c)} this strongly indicates that for $(n^M,c)=(15/16,0)$ the system has a double heteroclinic orbit. We can confirm this by computing the energy level $H(E,w)=0$ as shown in Figure \ref{fig:hamiltonian}.
        \end{proof}
		\begin{center}
	        \includegraphics[width=.7\textwidth]{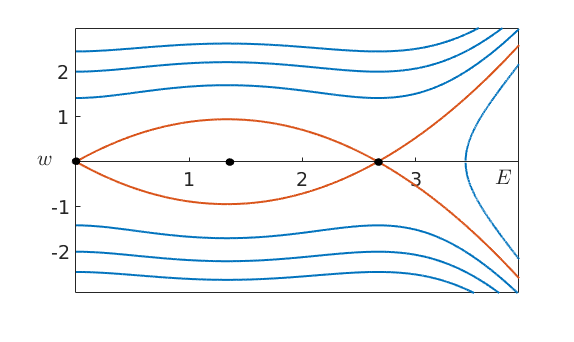}
	        \vspace*{-5mm}
	        \captionof{figure}{Energy levels of \eqref{eq:Karmafast} when $c=0$ showing a double heteroclinic orbit when $H(E,w)=0$ (orange).}\label{fig:hamiltonian}
	   \end{center}
        
        Before we continue illustrating the geometric ideas behind the proof of the second part of Lemma \ref{lemma:heteroclinics}, we want to make a remark regarding the previous construction.
        
        \begin{remark}
            In Figure \ref{fig:heteroclinic(n,c)} we can see the curve of heteroclinic orbits for different values of $M$. In particular we can see the insensitivity of the wave-front velocity with respect to the slow variable $n$ when $M\gg 1$ which is one of the important advantages mentioned in \cite{Karma93} of the Noble and Karma model over FitzHugh-Nagumo.
        \end{remark}
        
        \begin{remark}
            In \cite{Deng91} Deng proved that in the FitzHugh-Nagumo model, under certain conditions, the perturbation of a double heteroclinic orbit in the full system can result in infinitely many front and back wave solutions with an arbitrary number of oscillations. Although his results are not directly applicable in our situation as we would have to adjust the slow variable nullcline to obtain two full system equilibria on the two saddle-type branches, the existence of a double fast subsystem heteroclinic orbit in the Karma model clearly indicates already the possibility of more complex travelling waves than just single pulses.
        \end{remark}
        
        \begin{proof}[Sketch of proof of (ii).]
            We have seen in the previous part that the unstable manifold of $p_2$ and the stable manifold of $p_0$ connect uniquely for $c=c_{min}\approx 0.707$. Nevertheless, for $c>c_{min}$ we find a negatively invariant set enclosed by the $E$-axis, the stable manifold of $p_0$ and unstable manifold of $p_2$ and the vertical segment connecting them at $E=1$ as shown in Figure \ref{fig:heteroclinic(E,w)}. Since we know there are no further equilibria in this set and therefore also no limit cycle we can apply the Poincar\'e-Bendixson Theorem to obtain that the stable manifold of $p_0$ converges for $t\to-\infty$ to $p_2$ through the center manifold giving rise to further heteroclinic connections from $p_2$ to $p_0$.
        \end{proof}
        \begin{center}
            \includegraphics[width=0.7\textwidth]{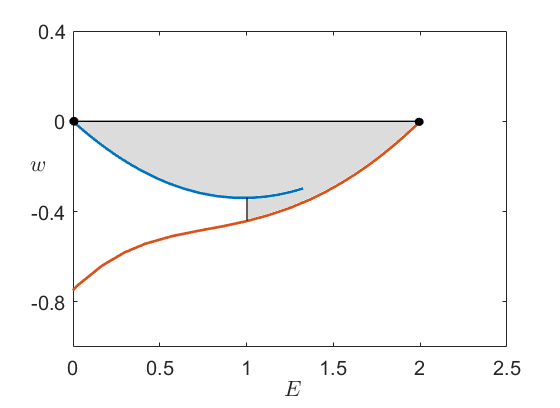}
            \captionof{figure}{Phase plane $(E,w)$ of the fast subsystem for $n=1$ and $c>c_{min}$. We can see the unstable manifold of $p_0$ (blue) and the stable manifold of $p_2$ (orange) and the negatively invariant set enclosed by them (grey).}\label{fig:heteroclinic(E,w)}
        \end{center}
        
        \begin{proof}[Sketch of Proof of Theorem \ref{th:KarmaTV0} (continued).]
            We can now easily construct a singular candidate orbit combining the slow and fast segments. Starting at the origin as the resting state we can jump to $p_2$ by a fast fibre where we follow the slow flow upwards. Since we jumped with $c\approx 1.77$ we cannot jump until we reach the fold point at $n=1$. Using the additional fast fibres we identified above we are able to jump back to $p_0$ and there follow the slow flow towards the origin.
        \end{proof}
        	
    	\begin{theorem}
    	    The homoclinic candidate orbit found in the singular limit $\varepsilon=0$ of equations \eqref{eq:KarmaPDE} can be perturbed to a homoclinic solution of the full system with $\varepsilon>0$.
    	\end{theorem}
        \begin{proof}[Idea of proof]
            The transition from the singular limit to the regular case can be done analogously to Section \ref{sec:ODE}. Away from $E=1$ where the system is not smooth and the non-hyperbolic fold point $(2,0,1)$ we can apply Fenichel's Theory (theorems \ref{th:fenichel1} - \ref{th:fenichel3}) to obtain the corresponding orbit in the regular case. Again, we can extend the orbits for $E=1$ by continuity since we know that we are away from the critical manifold and finally the fold point can be analysed using geometric blow-up as introduced in Appendix \ref{sec:blowup}.
        \end{proof}
        
        We recall that in the FitzHugh-Nagumo model a travelling wave will jump to a fast fibre directly from the normally hyperbolic part of the critical manifold. We have now shown that in contrast to that a pulse solution for Karma model needs the jump segments generated by the fold point through the center manifold. This is a key difference between the two models. It results in a fixed position of the wave back and a slower repolarization than depolarization rate which Karma already identified as important properties for cardiomyocytes (see \cite{Karma93}).
            
\section{Numerical simulations} 
    In this section we simulate the full PDE systems with a focus on the Karma model. In particular, we want to interpret the numerical simulations in relation to the analysis presented above in order to understand the PDE dynamics~\cite{Kuehn19} we can actually observe. For this we will use the parameter values $\varepsilon=10^{-2}$, $D=1$, $M=4$, $n_B=0.5$ and $I=0$ except explicitly mentioned otherwise. Figure \ref{fig:Sim} shows the evolution of the system initialised with a bump function centered at $x=50$.\\
    \\
    \includegraphics[width=0.5\textwidth]{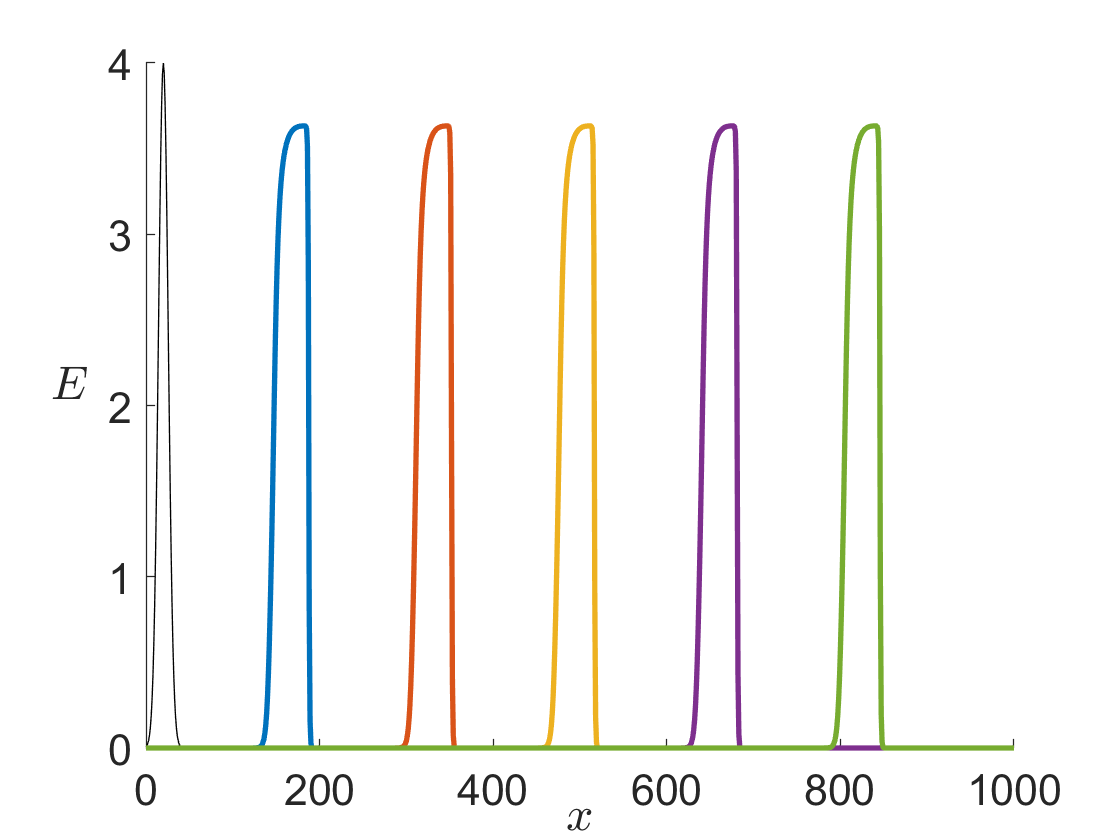}
    \includegraphics[width=0.5\textwidth]{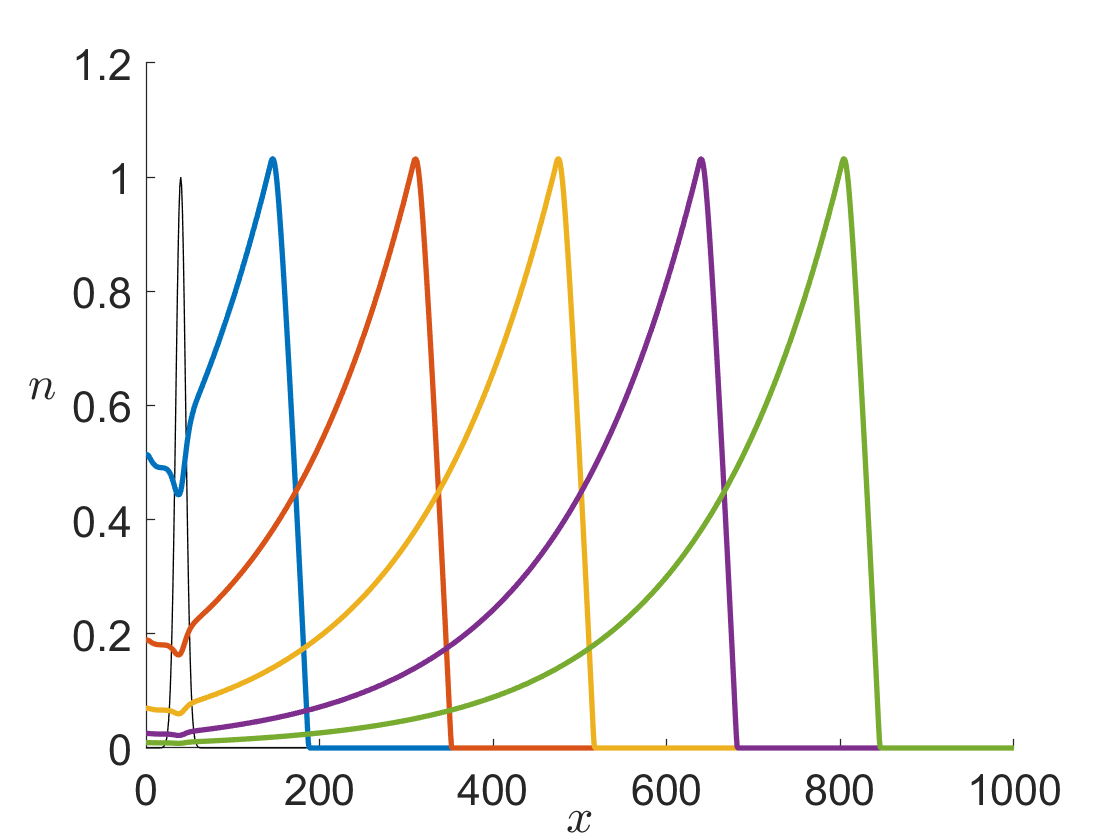}
    \captionof{figure}{\label{fig:Sim}Simulation of the Karma PDE model \eqref{eq:KarmaPDE} for the shown initial conditions (black) at $t=100,200,300,400 \text{ and } 500$.\\}
    
    For the Karma model, we see from Figure~\ref{fig:Sim} that in fact for a big enough region in $x$ the dynamics converge to a travelling pulse as we have found analytically. Since the simulations converge to a travelling wave given an arbitrary initial profile it (most likely) follows that the travelling pulse is at least locally asymptotically stable and that it does have a substantial basin of attraction. We have not proven the local asymptotic stability analytically here but this would be an interesting point in future work as it is well-known that the FHN PDE has wide parameter ranges, where stable pulses occur and where geometric techniques allow us to prove stability~\cite{Jones84,Jones91}.\\
    
    As a comparison, Figure \ref{fig:FHNSim} shows a similar simulation for the FitzHugh-Nagumo model~\eqref{eq:FHN} with $\varepsilon=10^{-2}$, $D=1$ and $I=0$.\\
    \\
    \includegraphics[width=0.5\textwidth]{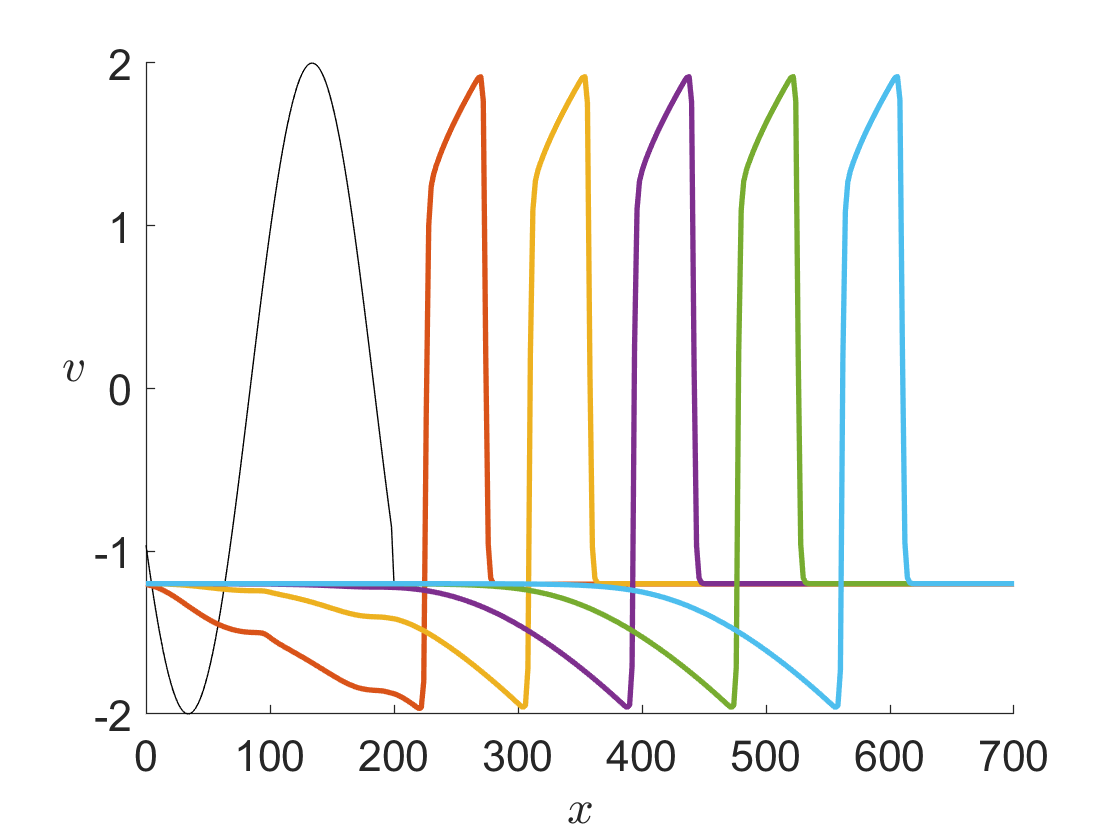}
    \includegraphics[width=0.5\textwidth]{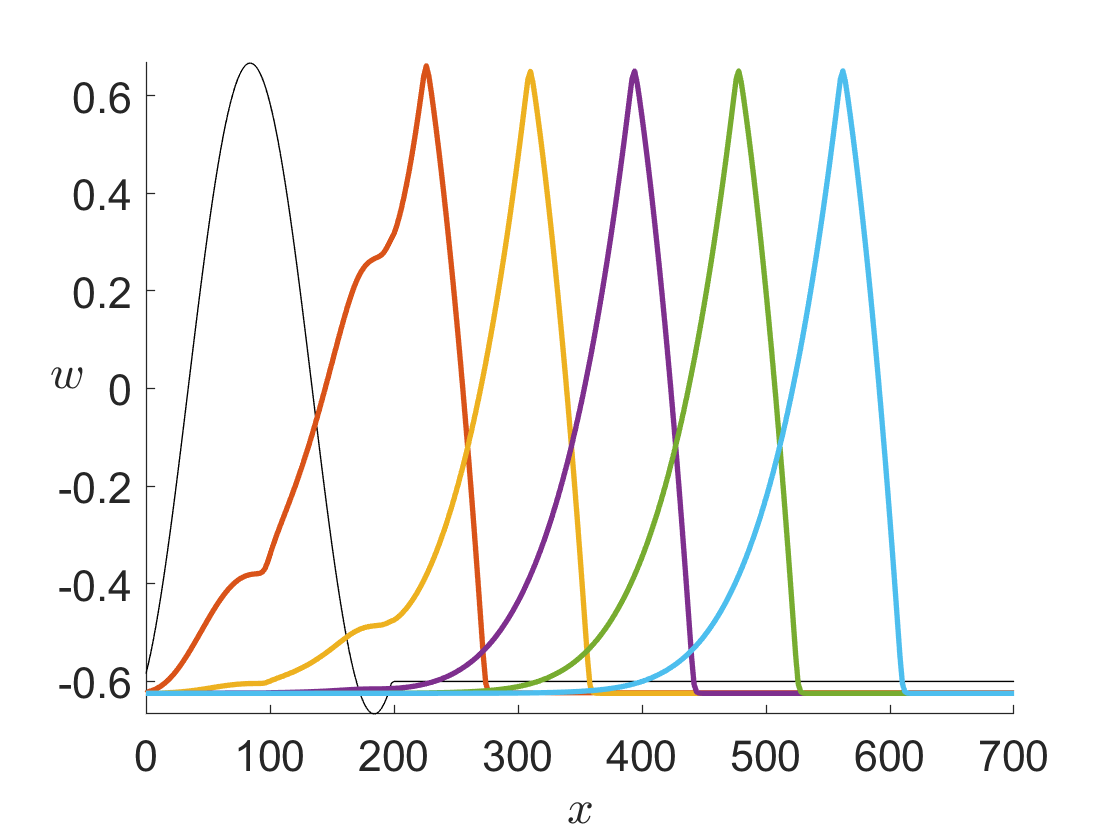}
    \captionof{figure}{Simulation of \eqref{eq:FHN} for the shown initial conditions (black) at $t=100,200,300,400 \text{ and } 500$.\\}\label{fig:FHNSim}
    
    At first sight we see that a big difference between Karma and FitzHugh-Nagumo is the hyperpolarization present only in the second model. Although there are heart tissues which show hyperpolarization, if we want to model e.g. ventricular cells the representation in the Karma model is notably more accurate. Furthermore we recall that the repolarization jump of the travelling wave we constructed in the previous section is ignited differently in both models, once on the fold point and once on the hyperbolic part of the manifold. Figure \ref{fig:TVPhase} shows that this is the case as well for the limit wave in the full PDE model. As stated before this is the reason for the slower recovery rate in the Karma equations which gives us a key difference between both models.\\
    \\
    \includegraphics[width=0.5\textwidth]{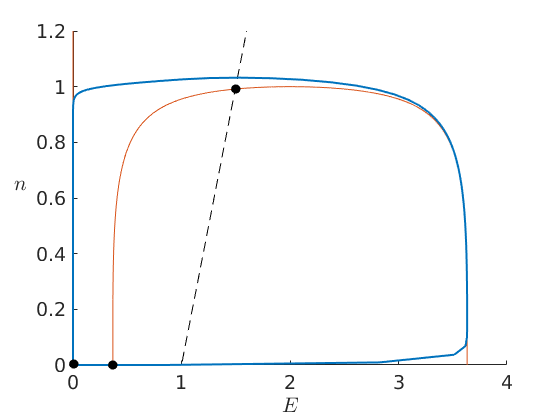}
    \includegraphics[width=0.5\textwidth]{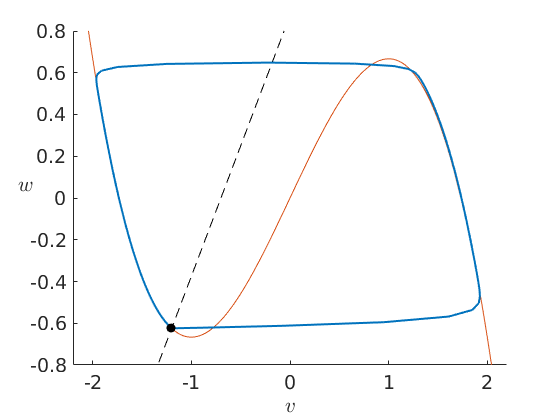}
    \captionof{figure}{Projection of the PDE solution for $t=500$ onto the $(E,n)$-plane in the Karma model (left) and the $(v,w)$-plane in the FitzHugh-Nagumo model.\\}\label{fig:TVPhase}

    To finish the numerical analysis we want to take a closer look at the effect of other parameters involved in the models 
    and look first at $\varepsilon$. We start with the Karma model. Following the values introduced in \cite{Karma93, Karma94} we have chosen for our simulations $\varepsilon=10^{-2}$ as our basis value. In addition, to make sure that the analysis above holds and we have in fact a travelling pulse solution, we need $\varepsilon$ to be small enough. Increasing $\varepsilon$ shows that already for $\varepsilon=0.08$ the travelling pulse dynamics seems to break down. Therefore, we will focus on smaller values of $\varepsilon$. By simulating the model with lower values we notice that, as expected, $n$ becomes slower as we decrease $\varepsilon$ so that the pulses for $E$ as well as $n$ elongate (see Figure \ref{fig:SimEpsilon}). Further we observe in the right panel that the convergence speed towards the travelling pulse is much slower for smaller $\varepsilon$. Nevertheless the wave speed appears to stay unchanged for different values of $\varepsilon$. Since we analytically demonstrated a geometric construction for the existence of the travelling pulses taking the wave speed $c$ as a parameter we would in fact expect changes in $c$ of order $\varepsilon$ with $c$ converging to the constant value $\approx 1.77$ as $\varepsilon\to 0$. It is also intuitively clear from a biological point of view that the wave speed should depend on the properties of the medium, e.g. the diffusion $D$, but be quite independent of the cells recovery speed.\\
    \\
    \includegraphics[width=0.5\textwidth]{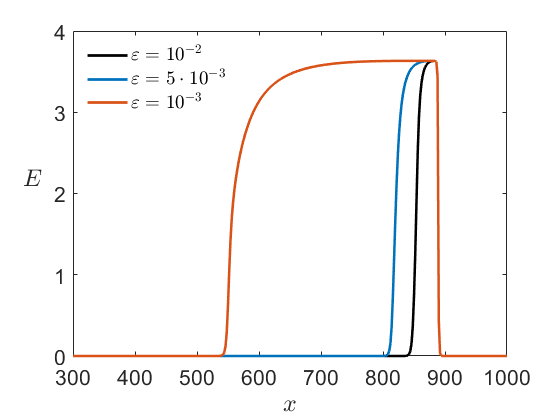}
    \includegraphics[width=0.5\textwidth]{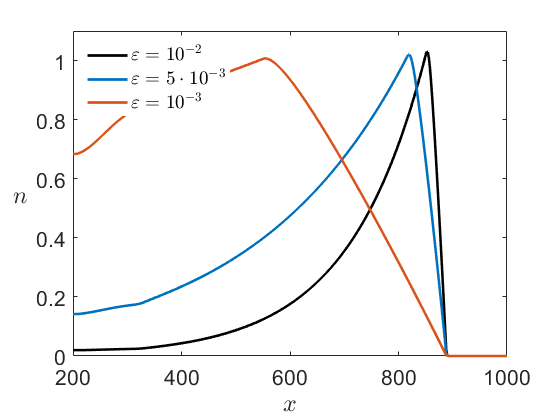}
    \captionof{figure}{Time shot of simulations of \eqref{eq:KarmaPDE} for multiple values of $\varepsilon$.\\}\label{fig:SimEpsilon}
    
    Again we can compare this with the effects of varying $\varepsilon$ in the FitzHugh-Nagumo model shown in Figure \ref{fig:FHNSimEpsilon}. Overall the effect of varying $\varepsilon$ observed in both models is similar. Nevertheless, for $\varepsilon=10^{-3}$ we find a change in the wave speed in the FitzHugh-Nagumo model while, as mentioned above, is not visible for the Karma model.\\
    \\
    \includegraphics[width=0.5\textwidth]{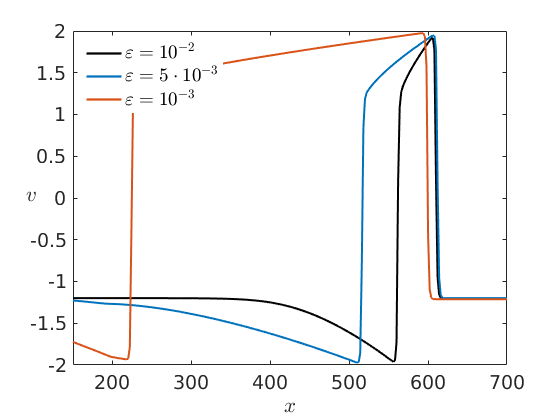}
    \includegraphics[width=0.5\textwidth]{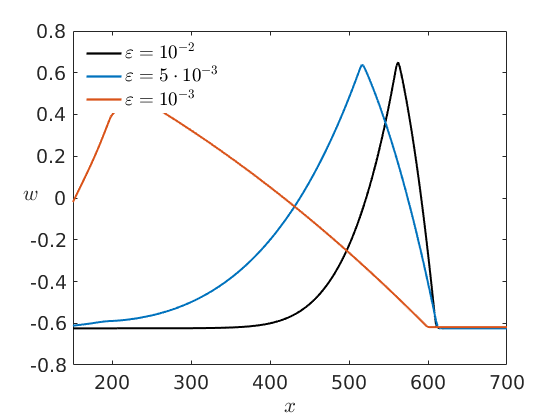}
    \captionof{figure}{Time shot of simulations of \eqref{eq:FHN} for multiple values of $\varepsilon$.\\}\label{fig:FHNSimEpsilon}
    
    We now want to consider the effects of different diffusion coefficients $D$ again starting with the Karma model. As mentioned before, we use as basis value for the diffusion $D=1$ for simplicity although the value used in \cite{Karma94} is $2.75$. In particular, we would like to make sure that $D\in\mathcal{O}(1)$ so that the previous analysis applies. Specifically for our model with $\varepsilon=10^{-2}$ our simulations lead to assume that $D>0.11$ since otherwise the pulse seams to disappear. In Figure \ref{fig:SimD} we consider three different simulations starting with the same initial conditions for different diffusion coefficients in the range of interest. We see that in this case the wave velocity is as expected strongly affected. An increase in the diffusion rate leads to higher wave velocity. Furthermore we also see that a bigger diffusion coefficient also results in a slightly longer pulse.\\
    \\
    \includegraphics[width=0.5\textwidth]{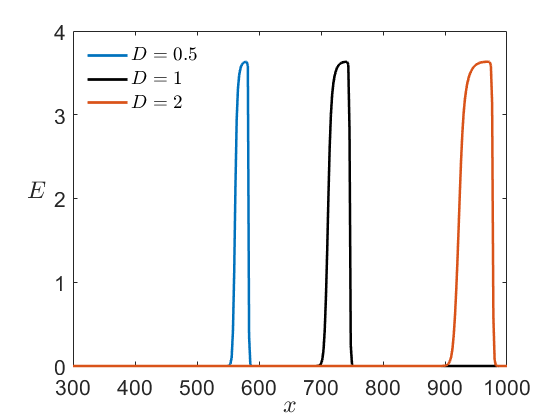}
    \includegraphics[width=0.5\textwidth]{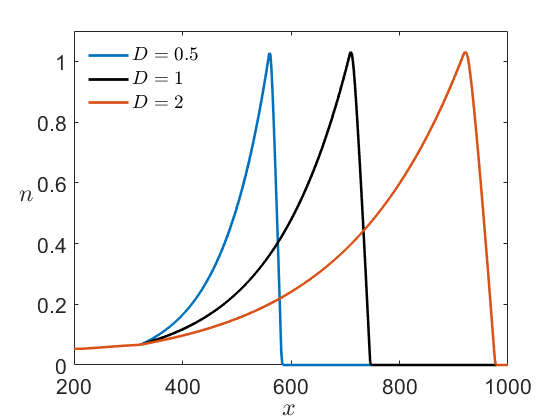}
    \captionof{figure}{Time shot of simulations of \eqref{eq:KarmaPDE} for multiple values of $D$.\\}\label{fig:SimD}
    
    In the corresponding simulation of the FitzHugh-Nagumo in Figure \ref{fig:FHNSimD} we see that the effects of different diffusion coefficients on both models are equivalent.\\
    \\
    \includegraphics[width=0.5\textwidth]{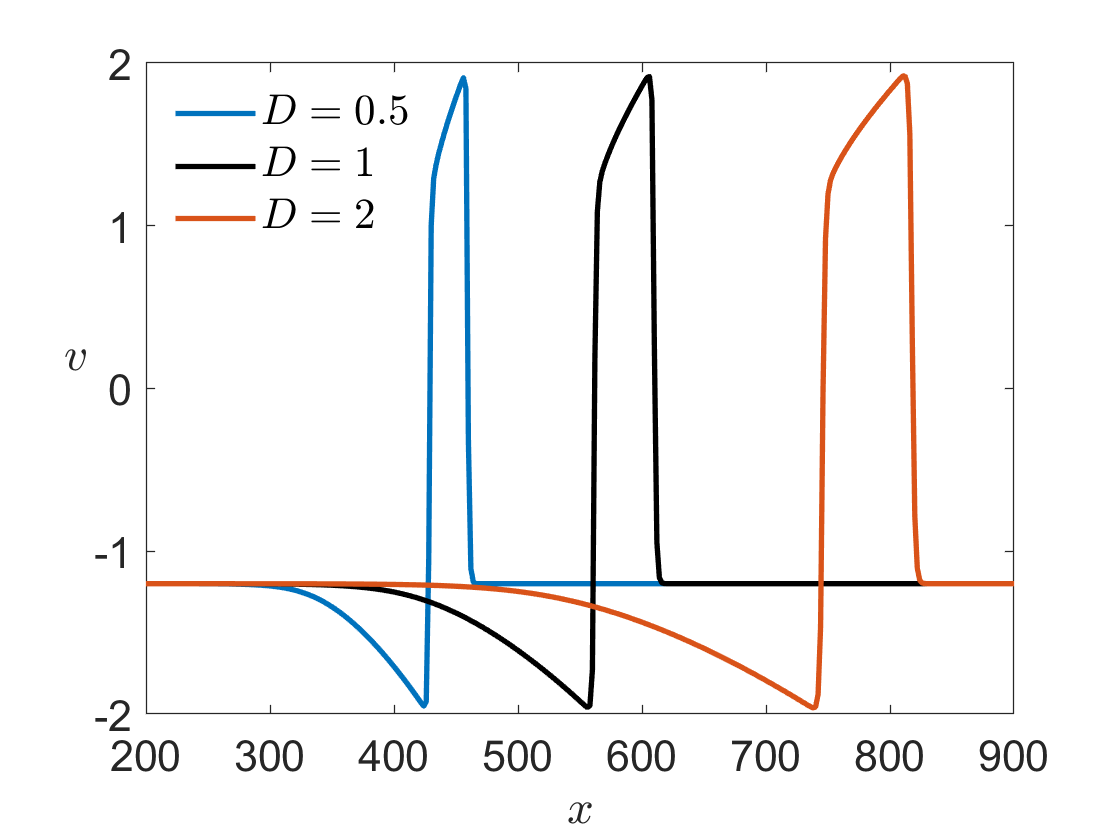}
    \includegraphics[width=0.5\textwidth]{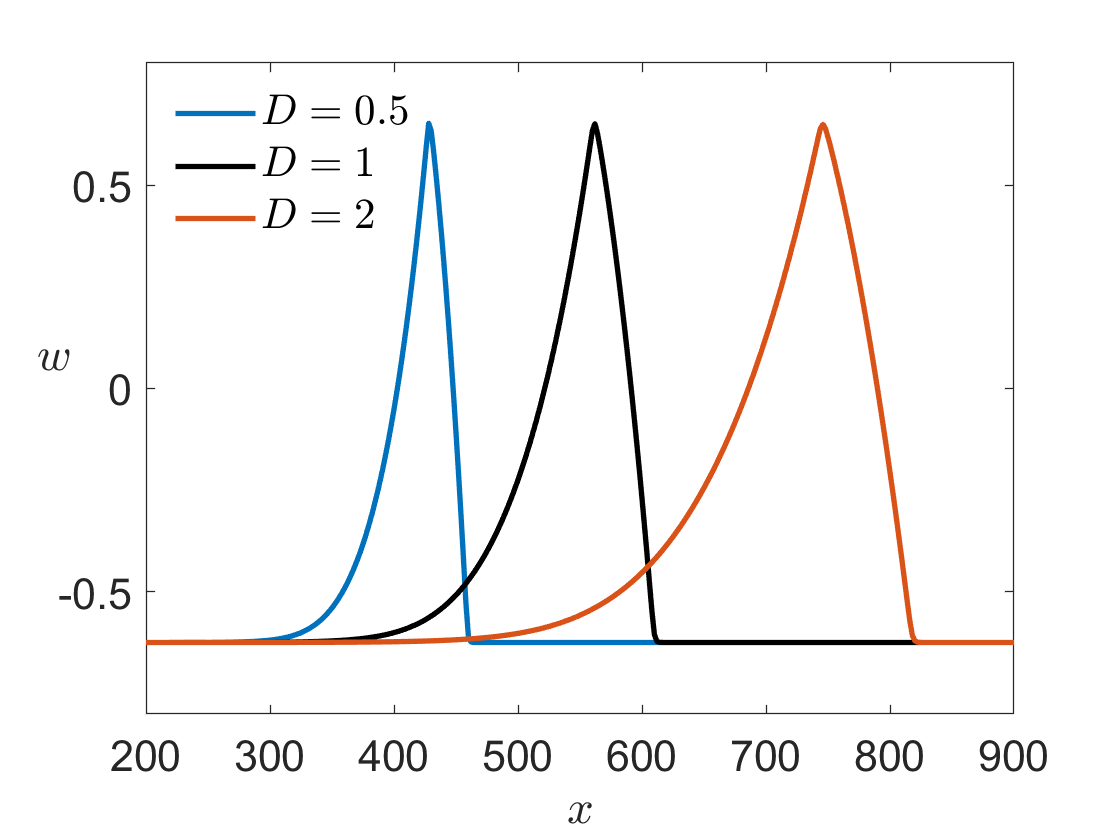}
    \captionof{figure}{Time shot of simulations of \eqref{eq:FHN} for multiple values of $D$.\\}\label{fig:FHNSimD}
    
    Similarly we can look at the control parameters $M$ and $n_B$ specific to Karma which are not fixed a priori. Using as a starting point again the values introduced by Karma in \cite{Karma93, Karma94} we follow the range of interest for the parameter $M$ that from modelling point of view varies from $M=4$ up to $M=30$. Even so, a higher or lower value does not qualitatively change the dynamics of the system. In Figure \ref{fig:SimM} we see that $M$ has almost no effect on the dynamics of the slow variable $n$ but controls the sharpness of the pulse for $E$. From biophysical modelling point of view this means that $M$ controls the sensitivity of the voltage $E$ with respect to the gating variable $n$.\\
    \\
    \includegraphics[width=0.5\textwidth]{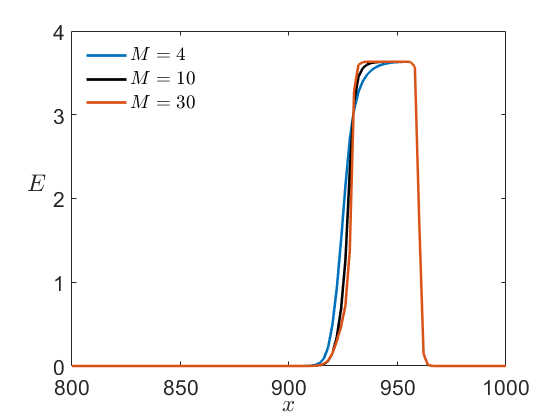}
    \includegraphics[width=0.5\textwidth]{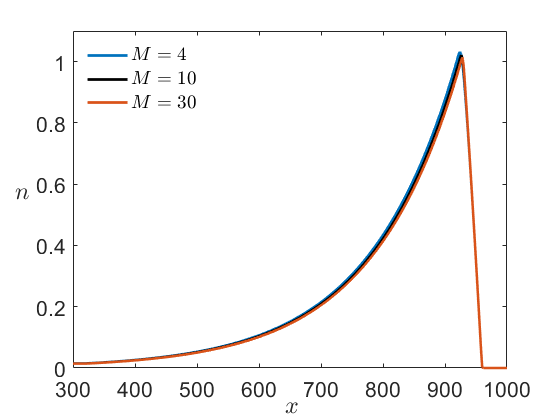}
    \captionof{figure}{Time snapshot of simulations of \eqref{eq:KarmaPDE} for multiple values of $M$.\\}\label{fig:SimM}
    
    On the other hand we know that $0<n_B<1$ and, more precisely, we expect to normally encounter values lying between $0.3$ and $0.8$. In contrast to the previous case, if we allow $n_B>1$ then the unstable equilibrium changes stability and the system becomes bistable giving rise to completely different dynamics. Focusing on the range suggested by Karma we find that the parameter $n_B$ determines the position of the wave back by controlling the speed of the slow subsystem. The higher $n_B<1$ the slower is the slow variable and therefore the longer is the depolarisation pulse (see Figure \ref{fig:SimnB}).\\
    \\
    \includegraphics[width=0.5\textwidth]{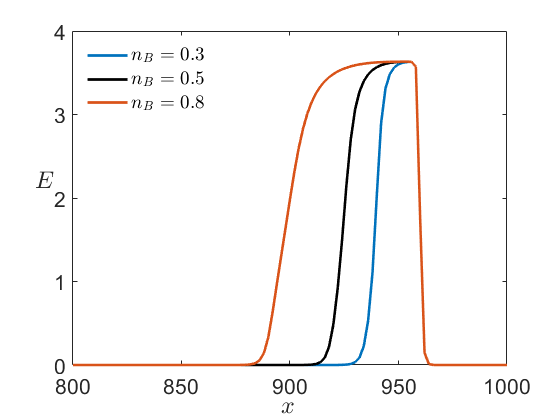}
    \includegraphics[width=0.5\textwidth]{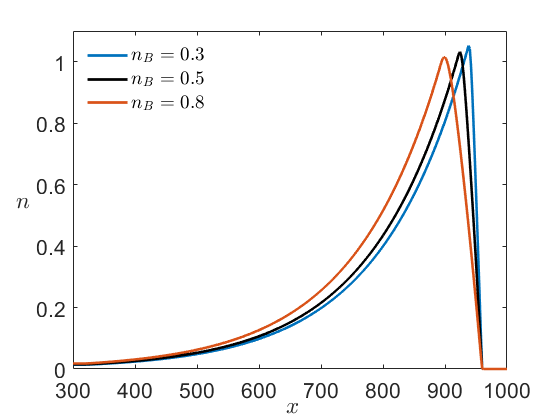}
    \captionof{figure}{Time snapshot of simulations of \eqref{eq:KarmaPDE} for multiple values of $n_B$.\\}\label{fig:SimnB}
    
    Last we can look at the effect of a small external current $I$ in the Karma model. From the analysis of the ODE model in Section \ref{sec:ODE} we know that for $I=I_0\approx0.087$ the system undergoes a saddle-node bifurcation so we cannot expect to have equivalent dynamics in the PDE case after crossing this point either. Nevertheless we want to compare the system for $I<I_0$ since we expect to be able to extend the analysis above in this range. In Figure \ref{fig:SimI} we see a time shot of the simulations for different values of $I$. At first sight we see that again the wave speed is changed where the higher the external current the faster the propagation speed of the wave. We can also see that the base line is no longer 0 but slightly higher approaching the fold point as $I\to I_0$ as we would expect. For $I=0.08$ we start being able to see that by increasing the base line we also get a weak hyperpolarization after the main pulse which we also would expect analytically due to the shape of the critical manifold.\\
    \\
    \includegraphics[width=0.5\textwidth]{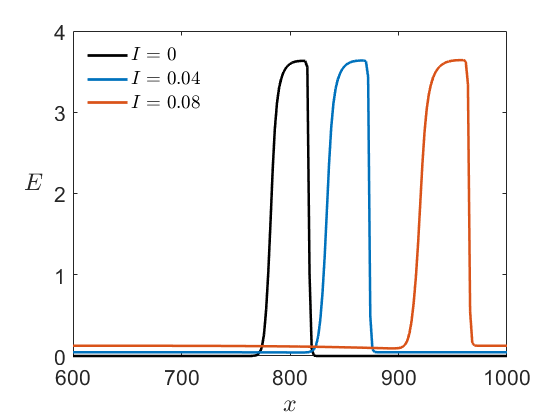}
    \includegraphics[width=0.5\textwidth]{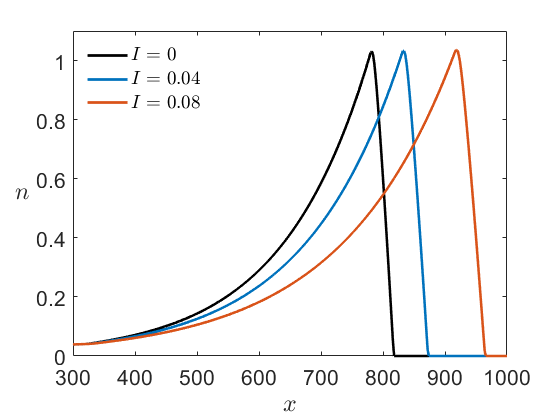}
    \captionof{figure}{Time shot of simulations of \eqref{eq:KarmaPDE} for multiple values of $I$.\\}\label{fig:SimI}
    
    Furthermore we can shift the waves such that the wave fronts coincide and see that the wave profile is also affected by the external current (see Figure \ref{fig:SimIprofile}). Although the effect is not as noticeable as the different wave speed we see that in addition to the higher base line we also have slightly longer pulses for higher incoming current.\\
    \\
    \includegraphics[width=0.5\textwidth]{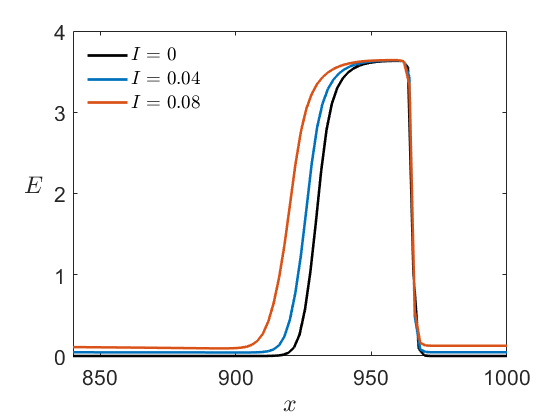}
    \includegraphics[width=0.5\textwidth]{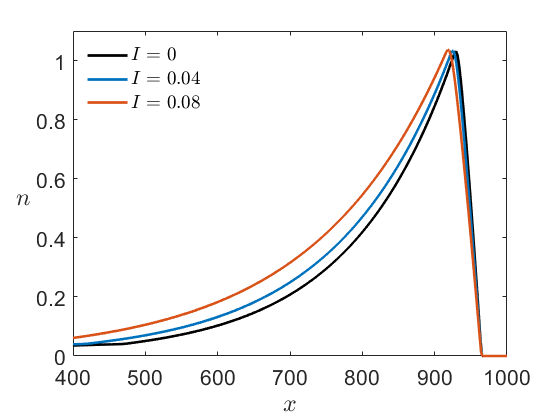}
    \captionof{figure}{Profile of the travelling pulses in the Karma model \eqref{eq:KarmaPDE} for multiple values of $I$.\\}\label{fig:SimIprofile}

\section{Discussion}
    We presented a systematic analysis and comparison of a polynomial version of the Karma model \cite{Karma93, Karma94} with the FHN model \cite{FitzHugh55, FitzHugh60, FitzHugh61} motivated by applications to modelling excitable behaviour in cardiomyocytes with regard to individual cells as well as cell populations. We started by considering their pure ODE versions. In this setting we noticed that Karma as well as FitzHugh-Nagumo present similar behaviours showing in both cases exactly three parameter regimes for the input current. When $I$ is sufficiently small the dynamics converge to a stable resting state while in the middle range of $I$ both models oscillate following a globally stable limit cycle. Finally, when $I$ is high enough any orbit converges to a stable equilibrium corresponding to a  depolarized state. Nevertheless, although both systems are qualitatively similar, there are also some likey important differences when applying them to model cardiomyocytes. First, in the Karma model the re-polarisation is much slower than the depolarisation because of the sharpness of the critical manifold while in the FHN model both processes are of the same order. Also, for a high external input $I$ the dynamics of the FHN model are still controlled by an ``S''-shaped critical manifold, in other words, depending on the initial conditions it is possible to undergo a depolarisation and re-polarisation before converging to the stable state. In contrast to that the Karma model does not allow any oscillation other than small fluctuations very close to the equilibrium given that in a reasonable regime for $\varepsilon$ the fixed point is a stable spiral. Yet, the biggest difference we see occurs when considering $I$ as a dynamic variable instead of a parameter. In the extended phase space we observe the high sensitivity to changes in $I$ when the system is oscillatory and most importantly the cusp singularity that arises when the two folds collide. Because of this differences it would be interesting for future work to look at the models with non-constant external current $I$.\\
    
    Next, we considered the spatially extended versions of the models and focused on travelling wave solutions in 1D without external current. This is motivated by our interest in modelling propagation of activity in populations of cardiomyocytes similar to \cite{Bernhard1, Bernhard2, Bernhard3}. We start by analysing the 1D PDE in the singular limit $\varepsilon = 0$ in order to study the existence of travelling wave solutions. Here, using similar techniques as used for FHN in \cite{Guckenheimer09} in addition to singular perturbation theory, we have demonstrated the existence of travelling pulses originating and converging to a fixed resting state. The first difference we have found comparing the Karma model to FHN is the insensitivity of the wave speed to different values of the slow variable. Furthermore, in contrast to FitzHugh-Nagumo, the wave back in the Karma model starts at the fold point for large parameter ranges resulting as in the ODE system in much slower re-polarisation than the previous depolarisation. All the analysis in this section has been restricted to $I=0$, therefore, as a future continuation of the work, it would be interesting to study if it is possible to extend the existence of travelling waves for $I>0$ and especially in the range where the ODE is oscillatory.\\
    
    Finally, we performed numerical simulations of the 1D PDE Karma model varying model parameters. As we would expect, the propagation velocity does not depend on the parameters controlling the reactivity of the cells but only on the parameters defining the medium, namely the diffusion coefficient $D$ and the background current $I$. On the other hand, while a change in $D$ or $I$ also affects the shape of the pulse we have observe that the main control over the shape is given by the reaction parameters $\varepsilon$, $M$ and $n_B$. Since all these are based only on observations of different simulations it would be another interesting avenue for future work to perform an even deeper analysis of the effect of the parameters on the travelling wave solutions.

\begin{appendices}


\section{Theory for multi-scale systems}\label{sec:theory}
	In this appendix we want to give an overview over the most important mathematical concepts and techniques used in this paper. We will start in Section \ref{sec:def} by introducing some definitions and basic results related to the analysis of dynamical systems and, in particular, those ruled by multiple time scales. For a more detailed introduction see \cite{Jones95, Kuehn15, Wiggins94}. Furthermore we will present in Section \ref{sec:fenichel} the singular perturbation theory developed by Fenichel in \cite{Fenichel71, Fenichel79} with its three key theorems and finally in Section \ref{sec:blowup} the blow-up technique is briefly introduced.
	
	\subsection{Definitions and notation}\label{sec:def}
		First, we introduce some important notions for the work with dynamical systems in general defined by the ODE
		\begin{equation}\label{eq:theory:ODE}
			\frac{dx}{dt}=\dot{x}=f(x), ~x=x(t)
		\end{equation}
		for $x:I\to\mathbb{R}^n$ with $I\subseteq \mathbb{R}$ an interval and some function $f\in C^r(\mathbb{R}^n,\mathbb{R}^n)$ ($r\geq 1$). We denote the flow induced by the differential equation \eqref{eq:theory:ODE} as $\phi_t(\cdot)$.
		
		\begin{definition} For a set $S\subset\mathbb{R}^n$ and a manifold $M\subset\mathbb{R}^n$ we have:
			\begin{itemize}
				\item $S$ is called \textbf{invariant} under the flow $\phi_t(\cdot)$ if $\phi_t(S)\subset S$ for all $t\in\mathbb{R}$.
				\item $S$ is called \textbf{positively invariant} if for all $p\in S$ it holds that $\phi_t(p)\in S$ for all $t\geq 0$.
				\item $S$ is called \textbf{negatively invariant} if for all $p\in S$ it holds that $\phi_{-t}(p)\in S$ for all $t\geq 0$.
				\item $M$ is called \textbf{locally invariant} under $\phi_t(\cdot)$ if for each $p\in\text{int}(M)$ there exists a time interval $I_p=(t_1,t_2)$ such that $0\in I_p$ and $\phi_t(p)\in M$ for all $t\in I_p$. In other words, the flow can only leave the manifold through its boundary.
			\end{itemize}
		\end{definition}
	
		We now continue with the systems showing multiple time scales. This is the case when some of the variables are much slower than others so we can identify two separate subsystems that move at different time scales.\\
		
		The setting we will be working with is a $(m,n)$-fast-slow system, this means we have an $m$-dimensional fast subsystem combined with $n$ further variables moving at a slower time scale. The complete system is given by the $(m+n)$-dimensional system of differential equations
		\begin{equation}\label{eq:theory:slow}
			\begin{aligned}
				\varepsilon\frac{dx}{dt}&=\varepsilon\dot{x}=f(x,y,\varepsilon)\\
				\frac{dy}{dt}&=\dot{y}=g(x,y,\varepsilon)
			\end{aligned}
		\end{equation}
		where $f\in C^r(\mathbb{R}^{m+n+1},\mathbb{R}^m)$ and $g\in C^r(\mathbb{R}^{m+n+1}, \mathbb{R}^n)$. For the equation \eqref{eq:theory:slow} the components of $x=x(t)\in\mathbb{R}^m$ are the \textbf{fast variables} and those of $y=y(t)\in\mathbb{R}^n$ are the \textbf{slow variables} of the system. The time scale separation is controlled by a small parameter $\varepsilon>0$, which provides the ratio between the \textbf{slow time scale} $t$ to the \textbf{fast time scale} $\tau :=t/\varepsilon$. Observe that \eqref{eq:theory:slow} describes the system with respect to $t$ and therefore we call it the \textbf{slow system}. If we write the system on the fast time scale $\tau=t/\varepsilon$, we obtain the \textbf{fast system} defined as
		\begin{equation}\label{eq:theory:fast}
			\begin{aligned}
				\frac{dx}{d\tau}&=x'=f(x,y,\varepsilon)\\
				\frac{dy}{d\tau}&=y'=\varepsilon g(x,y,\varepsilon).
			\end{aligned}
		\end{equation}
		
		When looking at fast-slow systems we are mostly interested in the case when $\varepsilon$ is very small and that brings up the question what happens at the limit $\varepsilon\to 0$ for a given system, the so-called \textbf{singular limit}. Although the slow and fast system \eqref{eq:theory:slow} and \eqref{eq:theory:fast} are equivalent since we only need a reparametrization of time to transform them into each other, their corresponding singular limit is not, so we have to differentiate between the \textbf{reduced system} or \textbf{slow subsystem}
		\begin{equation}\label{eq:theory:reduced}
			\begin{aligned}
				0&=f(x,y,0)\\
				\dot{y}&=g(x,y,0)
			\end{aligned}
		\end{equation}
		which we obtain taking the singular limit of the slow system \eqref{eq:theory:slow} and the \textbf{layer problem} or \textbf{fast subsystem}
		\begin{equation}\label{eq:theory:layer}
			\begin{aligned}
				x'&=f(x,y,0)\\
				y'&=0
			\end{aligned}
		\end{equation}
		as the limit for the fast system \eqref{eq:theory:fast}.
		The reduced system \eqref{eq:theory:reduced} consists of an algebraic constraint and the differential equation for the slow variables defining the so-called \textbf{reduced or slow flow}. In comparison, the layer problem is given by a differential equation for the fast variables where $y$ is a fixed parameter and the associated flow is called \textbf{fast flow}. The name arises as every value of $y$ defines an independent ``layer'' of the system.
		\begin{definition}
			The algebraic constraint in \eqref{eq:theory:reduced} defines the \textbf{critical manifold} $$C_0=\{(x,y)\in\mathbb{R}^{m+n}:f(x,y,0)=0\}$$
			The points contained in the critical manifold correspond exactly to the equilibrium points of the fast subsystem.
		\end{definition}
		
		\begin{remark}
			The equations \eqref{eq:theory:reduced} define the slow flow to be naturally restricted to the manifold $C_0$.
		\end{remark}
		
		In this setting we are able to decouple the fast and the slow dynamics in the system analysing them separately. Nevertheless, to obtain a global picture of the full system we need to combine the fast and the slow trajectories and we get the following definition.
		
		\begin{definition}
		    A \textbf{candidate orbit} is the image of a homeomorphism $\gamma:(a,b)\to\mathbb{R}^{m+n}$ with $a<b$ and a partition $a=t_0<t_1<\dots<t_k=b$ for some $k\in\mathbb{N}^+$ such that
		    \begin{itemize}
		        \item the image $\gamma((t_{i-1},t_i))$, $i\in\{1,\dots, k\}$ of each subinterval is a trajectory of either the fast or the slow subsystem
		        \item the image $\gamma((a,b))$ has an orientation that is consistent with the orientation of each trajectory $\gamma((t_{i-1},t_i))$, $i\in\{1,\dots, k\}$
		    \end{itemize}
		\end{definition}
	
	\subsection{Fenichel's Theory}\label{sec:fenichel}
	
		The following statements were first introduced by Fenichel in his paper ``\textit{Persistence and smoothness of invariant manifolds for flows}'' 1971 (\cite{Fenichel71}) and then applied to fast-slow systems 1979 in ``\textit{Geometric singular perturbation theory for ordinary differential equations}'' (\cite{Fenichel79}). Fenichel's work consists of three main theorems posed in a very general setting. Since we will not need this generality, we will only present the important results already applied to fast-slow systems. We will not prove any of the statements in this section, for the proofs see \cite{Fenichel71, Fenichel79, Kuehn15, Wiggins94}.\\
		
		To be able to understand the next theorem we first need some more properties of the critical manifold.
		
		\begin{definition}
			A subset $S\subset C_0$ is called \textbf{normally hyperbolic} if the Jacobian with respect to the fast variables $D_xf(x^*,y^*,0)\in\mathbb{R}^{m\times m}$ has no eigenvalue with zero real part for all $(x^*,y^*,0)\in S$.
		\end{definition}
		
		\begin{remark}
			The definition shows that $S\subset C_0$ is normally hyperbolic if and only if for every $(x^*,y^*,0)\in S$ it holds that $x^*$ is a hyperbolic equilibrium of the fast subsystem for $y=y^*$ i.e. $x^*$ is a hyperbolic equilibrium of $x'=f(x,y^*,0)$.
		\end{remark}
	
		Using this correspondence between fixed points of the layer problem and points of the critical manifold we can now analyse the stability properties of those equilibria and define the analogous concept for $C_0$.
		
		\begin{definition} Let $S\subset C_0$ be a normally hyperbolic set.
			\begin{itemize}
				\item $S$ is called \textbf{attracting} if for all $(x^*,y^*,0)\in S$ every eigenvalue of $D_xf(x^*,y^*,0)$ has negative real part i.e. for all $(x^*,y^*,0)\in S$ the corresponding equilibrium $x^*$ of the fast subsystem is stable for $y=y^*$.
				\item Similarly, $S$ is called \textbf{repelling} if all eigenvalues have positive real part i.e. the fixed points are unstable.
				\item If $S$ is neither attracting nor repelling, it is called \textbf{of saddle type}.
			\end{itemize}
		\end{definition}
		
		Finally, to measure the distance of the perturbed manifolds we will use the following metric.
		
		\begin{definition}
			The \textbf{Hausdorff distance} $d_H$ between to nonempty sets $V,W\subset\mathbb{R}^{k}$ is defined by
			$$d_H(V,W):=\max\left\{\sup_{v\in V} \text{dist}(v,W),~\sup_{w\in W}\text{dist}(w,V) \right\}$$
			where $\text{dist}(p,M):=\inf_{q\in M}||p-q||$ gives us the distance from a point $p\in\mathbb{R}^{k}$ to the set $M\subset\mathbb{R}^{k}$. In other words, the Hausdorff distance $d_H(V,W)$ defines the maximal distance between a random point in one set to the other set.
		\end{definition}

		\begin{theorem}[Fenichel's first Theorem, fast-slow version]\label{th:fenichel1}
			Let $S_0$ be a compact normally hyperbolic submanifold of the critical manifold $C_0$ of \eqref{eq:theory:slow} and $f\in C^r(\mathbb{R}^{m+n+1},\mathbb{R}^m),~g\in C^r(\mathbb{R}^{m+n+1},\mathbb{R}^n)$ for $1\leq r<\infty$. Then for $\varepsilon>0$ sufficiently small it holds that
			\begin{enumerate}[(F1)]
				\item There exists a locally invariant manifold $S_\varepsilon$ diffeomorphic to $S_0$,
				\item $S_\varepsilon$ has Hausdorff distance $\mathcal{O}(\varepsilon)$ from $S_0$,
				\item The flow on $S_\varepsilon$ converges to the slow flow as $\varepsilon\rightarrow 0$,
				\item $S_\varepsilon$ is $C^r$-smooth.
			\end{enumerate}
		\end{theorem}
	
		\begin{definition}
			The manifold $S_\varepsilon$ obtained as conclusion of Theorem \ref{th:fenichel1} is called a \textbf{slow manifold}.
		\end{definition}
	
		\begin{remark}
			$S_\varepsilon$ is usually not unique. Nevertheless, in regions lying at a fixed distance from $\partial S_\varepsilon$, all manifolds satisfying \textit{(F1)-(F4)} lie at a Hausdorff distance $\mathcal{O}(e^{-K/\varepsilon})$ from each other for some positive $K\in\mathcal{O}(1)$. For this reason, a representative of the manifolds is often called ``the'' slow manifold since, in most cases, it is arbitrary which representative to choose.
		\end{remark}
		
		By Theorem \ref{th:fenichel1} we know that, starting with a fast-slow system in the singular limit, if we perturb the equations by taking $\varepsilon>0$ sufficiently small the structure and behaviour of the critical manifold do not disappear. Instead, any compact subset of the manifold perturbs continuously to a slow manifold $S_\varepsilon$. The perturbation does not only preserve the topological structure of the critical manifold but the flow on the slow manifold is also $\varepsilon$-close to the original slow flow.

		\begin{theorem}[Fenichel's second Theorem] \label{th:fenichel2}
			Given the setting as in Theorem \ref{th:fenichel1}, the statements \textit{(F1)-(F4)} hold for the local stable and unstable manifolds if we replace $S_0$ and $S_\varepsilon$ by $W^{s,u}_{loc}(S_0)$ and $W^{s,u}_{loc}(S_\varepsilon)$ with
			\begin{equation*}
				W^{s,u}_{loc}(S_0)=\bigcup_{p\in S_0}W^{s,u}_{loc}(p)
			\end{equation*}
			In particular, $W^{s,u}_{loc}(S_\varepsilon)$ exist and furthermore $S_\varepsilon$ is normally hyperbolic and has the same stability properties with respect to the fast variables as $S_0$ (attracting, repelling or of saddle type).
		\end{theorem}
		
		Although the stable and unstable manifolds can only be defined for an equilibrium , with the help of Fenichel's second Theorem we are able to generalize this notion to fast-slow systems with $\varepsilon$ small enough. The stable and unstable manifolds $W^{s,u}_{loc}(S_0)$ that result from the union of those of the individual fixed points of the fast subsystem are not lost by the perturbation but instead we find that the topological as well as the analytical properties in new manifolds $W^{s,u}_{loc}(S_\varepsilon)$ remain similar to those of the original ones.
		
		\begin{theorem}[Fenichel's third Theorem]\label{th:fenichel3}
			We start again with the same setting as in Theorem \ref{th:fenichel1}. Then there exists a manifold $\mathcal{F}^u(p)$ for each $p\in S_0$ such that
			\begin{enumerate}[(a)]
				\item $\bigcup_{p\in S_0}\mathcal{F}^u(p)=W^u_{loc}(S_0)$,
				\item For $p\neq p'$ it holds that $\mathcal{F}^u(p)\cap\mathcal{F}^u(p')=\emptyset$,
				\item $\phi_{-t}(\mathcal{F}^u(p))\subseteq\mathcal{F}^u(\phi_{-t}(p))$,
				\item $\mathcal{F}^u(p)$ is tangent to $\mathcal{N}^u_p$ at $p$ with $\mathcal{N}_p^u$ the unstable component of the normal direction to $S_0$ (fast direction),
				\item There exist constants $C_u,\lambda_u>0$ such that if $q\in\mathcal{F}^u(p)$, then
				\begin{equation*}
					||\phi_{-t}(p)-\phi_{-t}(q)||<C_ue^{-\lambda_ut}
				\end{equation*}
				for every $t\geq 0$,
				\item $\mathcal{F}^u(p)$ is $C^r$ with respect to the base point $p$.
			\end{enumerate}
			The same conclusions \textit{(a)-(f)} with the obvious modifications hold for the family of manifolds $\mathcal{F}^s(p)$, e.g. replace $-t$ by $t$ in \textit{(c)} so that $\phi_t(\mathcal{F}^s(p))\subseteq \mathcal{F}^s(\phi_t(p))$. Furthermore, the foliation persists for $\varepsilon>0$ with all properties mentioned above and diffeomorphic to the foliation in the singular limit.
		\end{theorem}
	
		\begin{definition}
			The manifolds $\mathcal{F}^{s,u}(p)$ are called the \textbf{stable/unstable fibres} through $p$. 
		\end{definition}
	
		The families $\mathcal{F}^{s,u}$ build decompositions of the stable/unstable manifolds by submanifolds characterised by initial conditions approaching each other at the fastest rate in forward/backward time. They are therefore called \textbf{asymptotic rate foliation}. Since this foliation persists under perturbations we can conclude that the asymptotic behaviour on the stable and unstable manifolds stays unchanged for positive $\varepsilon$.
		
		
		These theorems are often referred to as \textbf{geometric singular perturbation theory} or \textbf{GSPT} as we do in this paper. Nevertheless, this term can also describe a wider compilation of geometric techniques for the analysis of singularly perturbed systems and consequently it is important to clarify if it refers to a specific result or to a hole branch of methods.
		
	\subsection{Geometric blow-up}\label{sec:blowup}
	    While Fenichel's Theory gives us the tools to analyse the critical manifold and its stable and unstable manifold whenever it is normally hyperbolic, the method of geometric blow-up enables us to investigate the points where the hyperbolicity is lost. The most common example are fold points as we see multiple times in this paper but also i.e. other types of bifurcation points like the cusp bifurcation appearing in Section \ref{sec:KarmaI}. This section aims at giving a basic understanding of the idea behind blow-up in general and applied to fold points in particular. It is not thought as a deep or complete study of this matter, see \cite{Kuehn15} for a more extensive introduction or the original papers \cite{Broer13, Dumortier78, Dumortier93, Krupa01a} for detailed statements and proofs.\\
	    
	    The method of geometric blow-up was first introduced in \cite{Dumortier78, Dumortier93} independently of systems with multiple time scale dynamics. The most simple case is given for a planar vector field $\dot z=F(z)\in\mathbb{R}^2$ as follows
	    
	    \begin{definition}
	        Let $S^1$ be the unit circle with the corresponding polar coordinate transformation
	        \begin{equation*}
	            \phi:S^1\times I\to\mathbb{R}^2~,~~~~\phi(\theta,r)=(r\cos\theta,r\sin\theta)
	        \end{equation*}
	        for some (possibly infinite) interval $I\subset\mathbb{R}$ with $0\in I$ and $\theta\in[0,2\phi)$ as parametrization of $S^1$. Then the \textbf{polar blow-up} $\hat F$ of a vector field $F\in C^\infty(\mathbb{R}^2)$ with $F(0)=0$ is defined by
	        \begin{equation}\label{eq:blowup}
	            \hat F(\theta,r):=(D\phi^{-1}_{(\phi,r)}\circ F\circ \phi)(\theta,r)
	        \end{equation}
	        for $r\neq0$ and by the continuous extension of \eqref{eq:blowup} to $r=0$.
	    \end{definition}
	    
	    \begin{remark}
	        The name polar blow-up clearly arises from the polar coordinate transformation used for $r>0$. Since the blow-up method can also be performed with a different coordinate transformation (see \textbf{directional blow-up} introduced in \cite{Kuehn15}), it serves as specification of the coordinate transformation that has been applied. Nevertheless, one can proof that both methods of blow-up, polar and directional, are equivalent up to a coordinate transformation.
	    \end{remark}
	    
	    \begin{definition}
	        Let $F$ be a $C^\infty$ vector field as above. We define the \textbf{(rescaled) polar blow-up} as
	        \begin{equation*}
	            \bar F:=\frac{1}{r^k}\hat F
	        \end{equation*}
	        with $k$ such that the derivatives of $F$ satisfy $D^k F=0$ and $D^{k+1} F\neq 0$. It is important to notice that this scaling does not change the qualitative structure of the orbits when $r>0$.
	    \end{definition}
	    
	    To illustrate the idea behind the blow-up method we present a simple planar example that can be desingularized using a polar blow-up transformation.
	    
	    \begin{example}
	        We consider the system
	        \begin{equation*}
	            \dot z=\begin{pmatrix}\dot{z_1}\\\dot{z_2}\end{pmatrix} =\begin{pmatrix}z_1^2-2z_1z_2\\z_2^2-2z_1z_2\end{pmatrix}=F(z_1,z_2)
	        \end{equation*}
	        It holds that $F(0,0)=0$ as well as 
	        \begin{equation*}
	            DF_{(0,0)}=\left.\begin{pmatrix}2z_1-2z_2&-2z_1\\-2z_2&2z_2-2z_1\end{pmatrix}\right|_{(0,0)}=\begin{pmatrix}0&0\\0&0\end{pmatrix}
	        \end{equation*}
	        so we see that the origin is a non-hyperbolic equilibrium. We now want to apply a blow-up transformation to construct a vector field with hyperbolic equilibria instead. The polar blow-up vector field is given by
	        \begin{equation*}
	            \hat F(\theta,r)=\begin{pmatrix}3r\cos\theta\sin\theta(\sin\theta-\cos\theta)\\\frac{1}{4}r^2(\cos\theta+3\cos(3\theta)+\sin\theta-3\sin(3\theta))\end{pmatrix}
	        \end{equation*}
	        so we can define the rescaled polar blow-up by $\bar F(\theta,r)=\frac{1}{r}\hat F(\theta,r)$. In this new vector field we find 6 hyperbolic equilibria on the circle $S^1\times{r=0}$ so we are now able to use the linear structure to describe the dynamics for $r$ small. Since the vector fields outside of $r=0$ are conjugated can use this information to describe qualitatively the structure of the orbits close to the origin in the original vector field.\\
	        \includegraphics[width=0.5\textwidth]{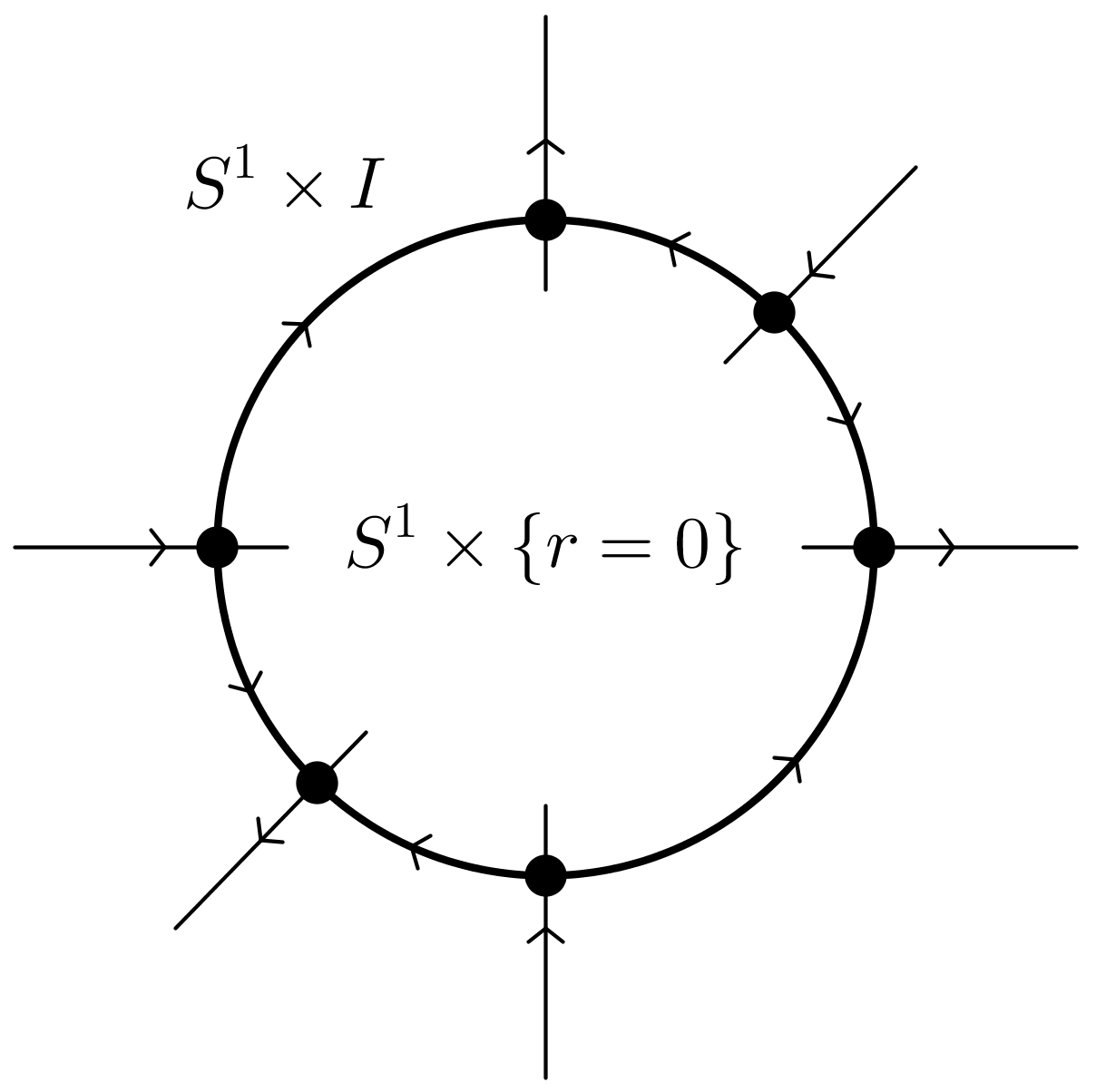}\hspace*{5mm}
	        \includegraphics[width=0.5\textwidth]{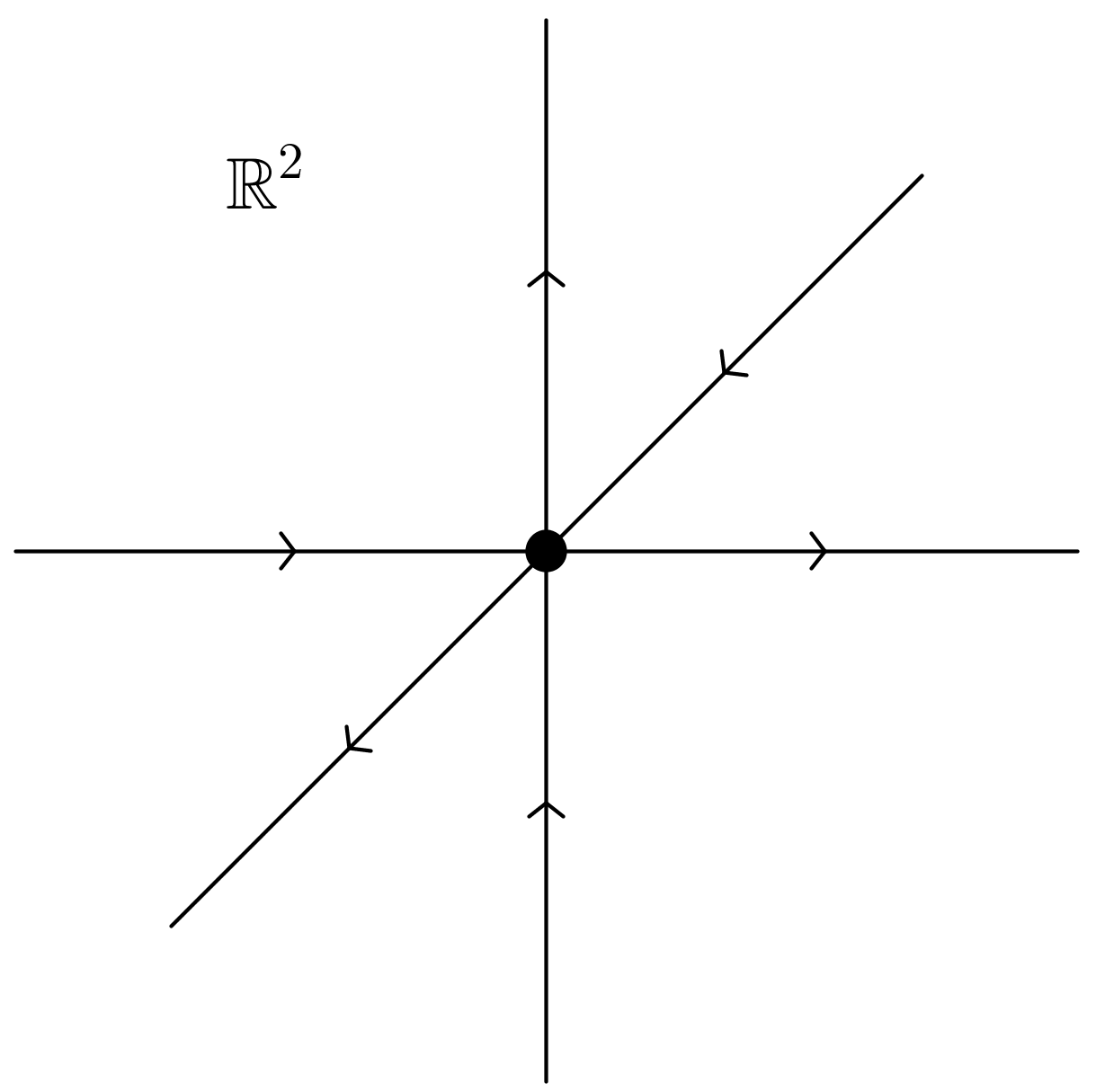}
	        \captionof{figure}{Blown-up vector field $\bar F$ (left) and original vector field $F$ (right). We see that we have transformed 1 non-hyperbolic equilibrium into a ring with 6 hyperbolic ones.}
	    \end{example}
	    
        We can generalize the definition above by allowing vector fields $F\in C^\infty(\mathbb{R}^n)$ in higher dimensions with the condition $F(0)=0$ unchanged.
	    
	    \begin{definition}
	        Let $S^{n-1}$ be the $(n-1)$-dimensional unit sphere in $\mathbb{R}^n$, then we consider the generalised polar coordinate transformation
	        \begin{equation*}
	            \phi:S^{n-1}\times I\to\mathbb{R}^n~,~~~~\phi(\bar z_1,\dots,\bar z_n,r)=(r\bar z_1,\dots,r\bar z_n)
	        \end{equation*}
	        where $\sum_{k=1}^n\bar z_k^2=1$ and some (possibly infinite) interval $I\subset\mathbb{R}$ with $0\in I$. Then we define the \textbf{polar blow-up} $\hat F$ of the vector field $F$ analogously to the case $n=2$ by
	        \begin{equation}\label{eq:blowup_n}
	            \hat F(\bar z_1,\dots,\bar z_n,r):=(D\phi^{-1}_{(\bar z_1,\dots,\bar z_n,r)}\circ F\circ \phi)(\bar z_1,\dots,\bar z_n,r)
	        \end{equation}
	        when $r\neq 0$ and the continuous extension of \eqref{eq:blowup_n} when $r=0$.
	    \end{definition}
	    
	    A further generalisation can be defined using instead of the classical polar coordinate transformation a generalized one as follows.
	    
	    \begin{definition}
	         Let $a_i\in\mathbb{N}$ for $i\in\{1,2,\dots,n\}$. In the same setting as the previous definition we consider the generalized polar coordinate transformation
	        \begin{equation*}
	            \varphi:S^{n-1}\times I\to\mathbb{R}^n~,~~~~\varphi(\bar z_1,\dots,\bar z_n,r)=(r^{a_1}\bar z_1,\dots,r^{a_n}\bar z_n)
	        \end{equation*}
	        and define therefore the \textbf{weighted} or \textbf{quasihomogeneous polar blow-up} by
	        \begin{equation}\label{eq:blowup_weighted}
	            \hat F(\bar z_1,\dots,\bar z_n,r):=(D\varphi^{-1}_{(\bar z_1,\dots,\bar z_n,r)}\circ F\circ \varphi)(\bar z_1,\dots,\bar z_n,r)
	        \end{equation}
	        when $r>0$ and the continuous extension for $r=0$.
	    \end{definition}
	    
	    \begin{remark}
	        It is clear, that the definition of the rescaled blow-up can be applied not only for two-dimensional vector fields but also to higher dimensional blow-ups as well as weighted polar blow-ups. In this last case we call it the \textbf{(rescaled) weighted polar blow-up}.
	    \end{remark}
	    
	    Finally we want to briefly present the results for the analysis of a fold point and first need to introduce some notation. The general setting of a fold point (w.l.o.g. located at the origin) is a $(1,1)$-fast-slow system as \eqref{eq:theory:fast} satisfying the conditions 
	    \begin{equation*}
	        f(0,0,0)=0~,~~~~f_x(0,0,0)=0
	    \end{equation*}
	    and w.l.o.g. 
	    \begin{equation*}
	        f_{xx}(0,0,0)>0~,~~~~f_y(0,0,0)<0~,~~~~g(0,0,0)<0.
	    \end{equation*}
	    We divide the critical manifold into the attracting and repelling branches $S_0^a$ and $S_0^r$ respectively. Analogously we define $S_\varepsilon^a$ and $S_\varepsilon^r$ as the attracting and repelling branches on the slow manifold, see Section \ref{sec:fenichel} for more details. Furthermore, for some $\rho>0$ small we define the sections
	    \begin{equation*}
	        \Delta^{\text{in}}:=\{(x,\rho^2):x\in J_{\text{in}}\}~~~~\text{and}~~~~ \Delta^{\text{out}}:=\{(\rho,y):y\in J_{\text{out}}\}
	    \end{equation*}
	    where $J_{\text{in}}$ and $J_{\text{out}}$ are intervals such that the transition map $\Pi:\Delta^{\text{in}}\to\Delta^{\text{out}}$ is well define as shown in Figure \ref{fig:fold}.
	    
	    \begin{center}
	        \includegraphics[width=0.7\textwidth]{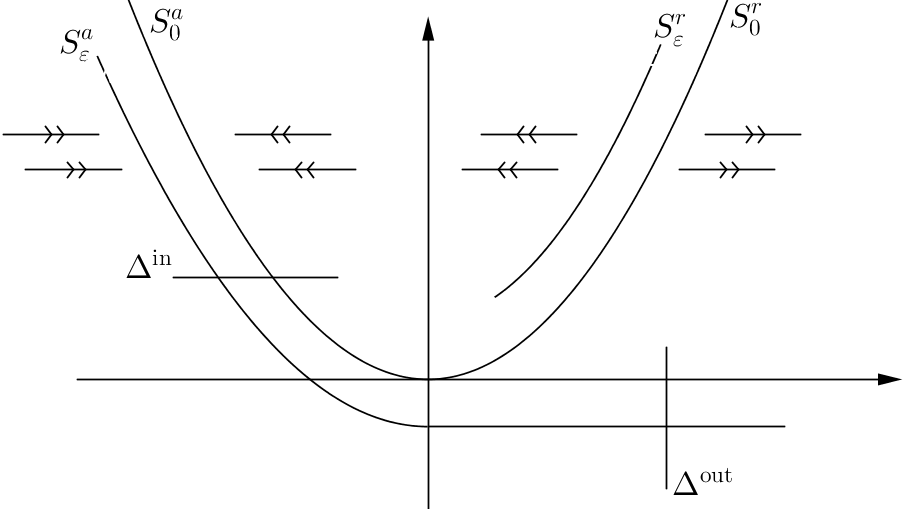}
	        \captionof{figure}{Phase plane near a fold point including the sections $\Delta^{\text{in}}$ and $\Delta^{\text{out}}$}\label{fig:fold}
	    \end{center}
	    
	    Now we can formulate the results of the flow near a fold point. The following theorem proofs that in fact the switch between the slow and fast flow at a fold point is conserved for $\varepsilon>0$ as shown in Figure \ref{fig:fold}. As mentioned in the introduction of this section we will not prove this theorem but only give the key ideas behind it, for a full proof see \cite{Krupa01a, Kuehn15}.
	    
	    \begin{theorem}\label{th:fold}
	        There exists $\varepsilon_0>0$ such that for all $\varepsilon\in(0,\varepsilon_0]$ the following holds:
	        \begin{enumerate}[(1.)]
	            \item The manifold $S_\varepsilon^a$ passes through $\Delta^\text{out}$ at $(\rho,H(\varepsilon))$ with $H(\varepsilon)\in\mathcal{O}(\varepsilon)$
	            \item $\Delta^{\text{in}}$ is mapped by $\Pi$ to an interval of size $\mathcal{O}(e^{-C/\varepsilon})$ for some $C>0$
	            \item The function $H(\varepsilon)$ has the asymptotic expansion
	            $$H(\varepsilon)=c_1\varepsilon^{2/3}+c_2\ln\varepsilon+c_3\varepsilon+\mathcal{O}(\varepsilon^{4/3}\ln\varepsilon)~~~~\text{ as }\varepsilon\to 0$$
	            \end{enumerate}
	    \end{theorem}
	    \textit{Key steps of the proof.} The idea to proof Theorem \ref{th:fold} is to extend the model to the 3-dimensional system with $\varepsilon$ a dynamic variable with derivative 0 and use that the origin is non-hyperbolic. Therefore, in this setting we can in fact apply a blow-up to desingularize the fold point. The correct transformation to achieve this is a weighted polar blow-up with the generalized polar coordinate transformation
	    \begin{equation*}
	        \varphi(\bar x,\bar y,\bar\varepsilon,r):=(r\bar x,r^2\bar y,r^3\bar\varepsilon)=(x,y,\varepsilon)
	    \end{equation*}
	    Once we have the rescaled vector field $\bar F$ it remains to find charts to represent the blow-up in local coordinates. Then we can analyze the dynamics and finally translate the results to the original ``blow-down'' vector field.\\

\end{appendices}


\newpage
\pagestyle{plain}

\end{document}